\documentclass[12pt, reqno]{amsart}

\usepackage{philstyle}

\def\cP{{\mathcal{P}}}

\def\cO{{\mathcal{O}}}

\def\RR{{\mathbb R}}

\def\Z{{\mathbb Z}}

\def\C{{\mathbb C}}

\def\raw{\rightarrow}

\def\tL{\tilde{L}}

\def\tx{{\tilde{x}}}

\def\tf{\tilde{f}}

\def\tg{\tilde{g}}

\def\tG{\tilde{G}}

\DeclareMathOperator{\Id}{Id}

\DeclareMathOperator{\Lip}{Lip}
\DeclareMathOperator{\esssup}{ess\: sup}
\DeclareMathOperator{\essinf}{ess\: inf}

\def\hJ{{\hat{J}}}

\def\hX{{\hat{X}}}

\def\hell{{\hat{\ell}}}

\def\vv{{\mathbf{v}}}

\def\One{\mathbbm{1}}

\def\floor#1{\lfloor #1 \rfloor}

\def\oi{\overline{i}}

\def\hx{\hat{x}}
\def\homega{\hat{\omega}}

\def\cSL{\mathcal{S}^{Lip}}
\def\cSr{\mathcal{S}^{r}}
\def\normL#1{\|#1\|_{Lip}}
\def\hpsi{\hat{\psi}}

\def\hphi{\hat{\phi}}
\DeclareMathOperator{\diag}{diag}

\def\bv{{\mathbf{v}}}
\def\bu{{\mathbf{u}}}
\def\bb{{\mathbf{b}}}
\def\bm{{\mathbf{m}}}

\def\ceiling#1{\left\lceil #1 \right \rceil}

\def\SPLk{\overline{\mathcal{S}}^{PL}_k}
\def\SPL{\overline{\mathcal{S}}^{PL}}
\def\TPLk{\overline{\mathcal{T}}^{PL}_k}
\def\TPL{\overline{\mathcal{T}}^{PL}}
\def\TPLkn\overline{{\mathcal{T}}^{PL}_{k, n}}

\def\bell{{\pmb{\ell}}}
\def\hbell{\hat{\bell}}
\def\hell{\hat{\ell}}
\def\oell{\overline{\ell}}
\def\obell{\overline{\bell}}

\def\bw{{\mathbf{w}}}
\def\hbw{{\hat{\bw}}}
\def\hw{{\hat{w}}}
\def\ha{\hat{a}}
\def\ow{\overline{w}}
\def\obw{\overline{\bw}}
\def\sw{\breve{w}}
\def\sbw{\breve{\bw}}

\def\oobw{\overline{\obw}}
\def\oow{\overline{\ow}}
\def\ob{\overline{b}}
\def\oomega{\overline{\omega}}
\def\oX{\overline{X}}
\def\oJ{\overline{J}}

\def\hb{\hat{b}}
\def\hbb{{\hat{\mathbf{b}}}}

\def\sw{\breve{w}}
\def\sz{\breve{z}}
\def\sx{\breve{x}}

\def\sJ{\breve{J}}
\def\sX{\breve{X}}
\def\sZ{\breve{Z}}

\def\pvl{\phi_{\bw,\bell}}

\DeclareMathOperator{\myspan}{span}

\usepackage{amssymb,amsmath,amscd,xspace}
\usepackage{enumerate}
\usepackage{graphicx}
\usepackage{bm}
\usepackage{xcolor}
\usepackage{soul}
\usepackage[normalem]{ulem}

\begin{document}

\title{Pinched Arnol'd tongues for  families of circle maps}

\author{Philip Boyland}
\address{Department of
    Mathematics\\University of Florida\\372 Little Hall\\Gainesville\\
    FL 32611-8105, USA}
\email{boyland@ufl.edu}

\begin{abstract} 
We prove  that  generically for a family of circle maps
\begin{equation*}
f_{b, \omega} (x) = x + \omega + b\, \phi(x)
\end{equation*}
with $\phi$ a piecewise linear forcing with $k>2$ breakpoints 
there is no pinching in any of its Arnol'd tongues. 
This is in contrast to a theorem of Campbell, Galeeva, Tresser, and Uherka
who showed  that with two break points
there are always multiple pinchpoints in its rational tongues. 
We also
prove that the absence of pinching is generic for Lipschitz
and  $C^r$ ($r>0$) forcing.
The family $f_{b, \omega}$
is used as a simple model for a periodically forced oscillator.
The rational tongue $T_{p/q}$ represents parameter values
where the system is mode-locked into a $p/q$-periodic response.
The pinching of the tongues to a point represents parameter values
where the system's periodic response is unstable to all perturbations
in the frequency  parameter $\omega$. The theorems in this paper
show that typically this type of instability does not occur
in the families under consideration.

\end{abstract}

\maketitle

\section{Introduction}
In \cite{arnold} Arnol'd introduced an example of a  parameterized
family of circle maps  and noted that it
provides a paradigm for describing  periodically forced oscillators.
It illustrates the common phenomenon of rational mode locking and 
irrational quasiperiodicity.
 The family,  defined by its lift as,
\begin{equation*}
\tf_{b, \omega} (x) = x + \omega + \frac{b}{2 \pi} \sin(2 \pi x)
\end{equation*}
has come to be called the Arnol'd family, or more commonly,
the standard family. In keeping with its physical origins
$\omega$ is sometimes called the intrinsic frequency, $b$ the coupling,
and $\frac{\sin(2 \pi x)}{2 \pi}$ the forcing.
A fairly large literature has grown up
around the family, both from a purely mathematical point of view 
and with a view to applications.
In addition, the bifurcation diagrams of many higher dimensional
systems exhibit the same kind of structures.

When $b<1$, the map $f_{b,\omega}$ is a diffeomorphism and 
the family's bifurcation diagram in the $(\omega, b)$-plane is analyzed 
using the rotation number $\beta$  of the diffeomorphism.
When $\beta = p/q$ this set of parameters has a sharp point
on the $\omega$-axis and opens out    
into a horn or tongue, hence the name ``resonance horn"
or ``Arnol'd tongue" (see Figure~\ref{firsttwo}, left).
 When $\beta$ is an irrational,   the set
is a line connecting the line $b=0$ to the line
$b=1$. Herman (\cite{herman1}) proved basic properties
of  these sets, in particular, for each $0<b\leq 1$ the 
horizontal lines of constant 
$b$  intersect the rational tongues in a nontrivial interval.

The dynamics of the piecewise linear (PL)
analog of the quadratic family on the interval,
the tent family, have been much studied. The analog on
the circle replaces the sine term with a
triangle wave.
To those not familiar with PL dynamics 
the resulting bifurcation diagram yields something
of a surprise;
 as the rational tongues with $q>2$ move upward to increasing $b$,
they  pinch down to a point before spreading out again to hit
the $b=1$ line (see Figure~\ref{firsttwo}, right). In the language of
this paper,  
Theorem $4$ from \cite{trouble} says
\begin{theorem}[Campbell, Galeeva, Tresser, and Uherka]\label{pinchmain}
Assume that $\phi$ is a PL, standard-like  forcing with two
break points and $\delta$ is the width of the interval where
$\phi' < 0$. In the $(\omega, b)$-bifurcation diagram of
the   family $f_{b,\omega}(x) = x + \omega + b\, \phi(x)$, for each $q$ the
$p/q$-tongue pinches in the $\ceiling{q \delta}-1$ places
at coupling values $b$ which are solutions in $(0,1)$ of the polynomial
equation
$$
 (1-y)^j (1 + \frac{\delta}{1-\delta}y)^{q-j} = 1
$$
for $1 \leq j \leq \ceiling{q \delta}-1$.
\end{theorem}
The triangle wave forcing corresponds to $\delta = 1/2$. Thus
its $(\omega, b)$-family has pinched $T_{p/q}$ tongues  for all $q>2$
with the number of pinch points of $T_{p/q}$ going to infinity
like $q$.
\begin{figure}\label{firsttwo}
\includegraphics[width=.49\linewidth]{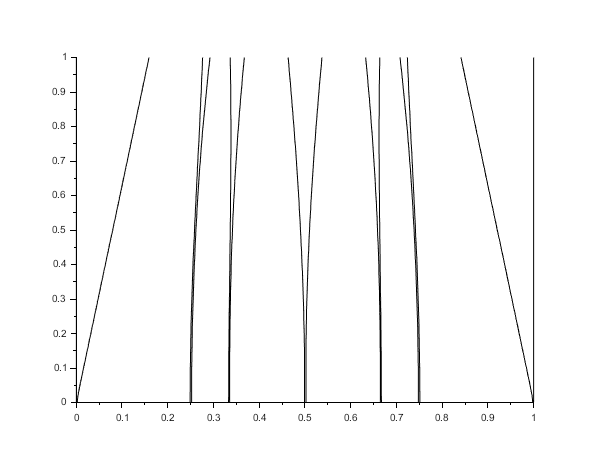}\hfill
\includegraphics[width=.49\linewidth]{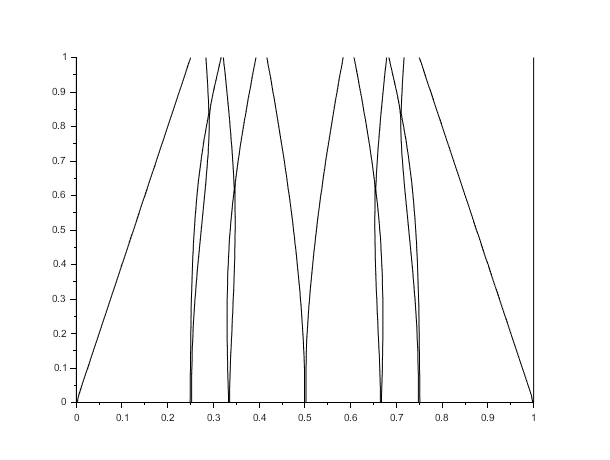}
\caption{Boundaries of Arnol'd $p/q$-tongues for $q\leq 4$. Left, standard
family; right, triangle wave family showing pinching. }
\end{figure}

The pinching in PL circle families with two break points
first appeared in the
literature in a paper by Yang and Hao \cite{first}. They
give explicit formulas for the edges of the pinched $p/q$-tongues
for $q\leq 6$ and $\delta=1/2$. Pinching
played prominent role in the major, comprehensive analysis of 
 classes of PL circle maps with two break points
by Campbell, Galeeva, Tresser, and Uherk
in \cite{trouble}. Meiss and Simpson (\cite{MS1, MS2, MS3, MS4}) and
Simpson (\cite{S2, S3, S4})
have made a thorough analysis of pinching in multiple
dimensional systems via the border collision bifurcation. 
Glendinning, Ma and Montaldi give multiple characterizations
of pinch points and the local scaling of the rotation numbers 
(\cite{paul}). Siyuan Ma has studied a family
of circle maps with four break points and symmetries in terms
of conjugation to rigid rotation which is equivalent to
pinch points (personal communication).
The book by  Mosekild and Zhusubaliyev considers pinched 
resonance regions and is a good resource for PL dynamics.
Various terminology has been used: the pinched Arnol'd tongues are
 sometimes called ``sausages'' or ``lens chains'' and the pinch points 
``shrinking points'' or  ``nodes''. The pinching of resonance regions
has been observed in models of various physical phenomenon
 (see \cite{S4} and \cite{book}  for references). 

Pinching in PL circle families with two breakpoints forces a
question: how common is  this amongst, say, families of the form
\begin{equation}\label{first}
\tf_{b, \omega, \phi}(x) = x + \omega + b\, \phi(x)
\end{equation}
for various classes of $\phi$'s?
A forcing $\phi$ is called \textit{standard-like} if the
family \eqref{first} behaves like the standard family in 
the sense that when $b<1$, $f_{b,\omega}$ is a homeomorphism,
and  when $b>1$, $f_{b,\omega}$ is not order preserving. The
genericity theorem proved here is:
\begin{theorem}\label{intromain}
For  standard-like forcing $\phi$ it is generic that 
the bifurcation diagram  of their $(\omega, b)$-family
 does not have any pinched tongues in each of these 
classes:
\begin{enumerate}[(a)]
\item PL with $k$ break points for each $k>2$
\item Lipschitz
\item $C^r$ with $r>0$
\end{enumerate}
\end{theorem}
There is a different topology for each class. For each $k$ the PL
forcings are topologized by the location of the break points 
and the slopes in the regions between them. The Lipschitz and
$C^r$ forcings use the standard Lipschitz and $C^r$ norms.

It is worth pointing out what such genericity results say
and don't say. Theorem~\ref{intromain} does not say that
no pinching happens in the given classes, just that it
doesn't happen for the topologically typical forcing, i.e. 
for a dense, $G_\delta$ set of forcings. The complement of
such a set can be dense. For example, Figure 6 in \cite{newref}
shows pinching for a family with four break points.
In addition, since the rotation
number of a circle homeomorphism is continuous in the 
$C^0$ topology, any forcing near one that pinches
will exhibit   "near pinching", and so pinching can empirically
look stable under perturbation of the forcing. 

 One thing is clear though from
Theorem~\ref{intromain}, the case of $k=2$ breakpoints is
special. There are various characterizations of a map
$f_{b,\omega}$ 
corresponding to a pinch point of $T_{p/q}$: $\tf_{b,\omega}^q = R_p$
where $R_p$ is translation by $p$ or
$f_{b,\omega}$ is conjugate to $R_{p/q}$ or all the breakpoints 
of $f_{b,\omega}$ are on periodic orbits or $f_{b,\omega}$
preserves an absolutely continuous invariant measure whose density
is a step function (Theorem~\ref{another} below as well as \cite{paul}). 
All the characterizations point
to what a special coincidence pinch points are. When $k>2$ there
 is sufficient freedom to perturb away this coincidence. The 
questions of genericity in terms of measure and higher dimensional
maps with more than two regions of definitions
deserve attention. 

Herman (\cite{herman1},  III.4.5)
 proved a result related to Theorem~\ref{intromain} and
our proofs
are modeled in part on his. He first observed that a $p/q$-tongue
pinches at $(\omega, b)$ if and only if $\tf_{b,\omega, \phi}^q = R_p$.
 His result
says:
\begin{theorem}[Herman]\label{introherman}
For the generic $C^r$ ($r \geq 0$) diffeomorphism $g$ of the circle
for all $\omega$ and $p/q$, $(R_\omega\circ \tg)^q \not = R_p$.
\end{theorem}
There are  similarities and  differences between
the theorems. Theorem~\ref{introherman} focuses on  
the diffeomorphism $g$ itself, not the
standard-like forcing $\phi$. One consequence is that
the topologies involved are different. In addition,
Theorem~\ref{introherman} doesn't consider the $b$ parameter.
On the other hand, Theorem~\ref{introherman}
holds for $C^0$-homeomorphisms $g$, and there is no $C^0$-version of
Theorem~\ref{intromain} here (see Remark~\ref{czero}) but there are Lipschitz
 and PL genericity results here.   Both $C^r$ theorems depend on Herman's
observation that when $\phi$ is a nonconstant trigonometric
polynomial, $(Id + \phi)^q\not = R_p$ for all $p/q$.

In addition, PL homemorphisms and more generally piecewise smooth
homeomorphisms of the circle have been studied by many authors starting
with  Herman (\cite{herman1}, section VI);
subsequent papers include \cite{liousse, liousse1, pl06,
pl14, pl15, pl18}.  Ghazouani shows that amongst PL circle
homeomorphisms with two affine pieces being Morse-Smale
is generic (\cite{ghaz}). Crovisier studied a PL family of circle maps
above the critical line (\cite{sylvain1, sylvain2}) as did
Alsed\`a and Ma\~nosas (\cite{als}). The PL 
case is included in sections 2 and 3 of \cite{boyland}.

The paper proceeds as follows. In Section~\ref{stdlike}
we characterize standard-like
forcings $\phi$. The $C^r$ and Lipschitz genericity 
theorems are proved in Section~\ref{main}.
They depend on the $C^r$-density lemma (Section~\ref{smooth})
and the  Lipschitz density theorem (Section~\ref{lipschitz}), respectively.
A good portion 
of the paper is spent proving the Lipschitz density theorem,
namely, given
$p/q$ and $n>1$ the collection of Lipschitz forcings with 
\begin{equation}\label{need}
\tf_{b, \omega, \phi}^q \not= R_p
\end{equation} 
 holding for all $\omega$ and $b\in [1/n, 1]$
is dense in the space of Lipschitz forcings.

Since trigonometric polynomials are not dense in $C^{Lip}(S^1)$
and indeed this space is not separable, the Lipschitz
density result has to be done ``by hand''. The proof
 has a number of steps. First,
a given Lipschitz $\phi$ is approximated by a $\hphi$ whose
derivative is a simple function with the form $\sum \hw_i I_{X_i}$ for a system
of weights $\hbw$ and a finite partition $\{X_i\}$ of $S^1$.
Then using the Implicit Function Theorem and some linear
algebra,  the weights are perturbed and an integral is taken
to  yield a standard-like $\psi$ for which \eqref{need} holds.

Sections~\ref{notlast} and \ref{last}  take up the PL case. 
The two main tools are the approximation
result proved for Lipschitz functions in Theorem~\ref{nonlin} and
an analysis of the combinatorics of the periodic orbits containing
the break points at pinching parameters. Finally, for completeness,
using the techniques developed in this paper we give an  alternative
proof  of Theorem~\ref{pinchmain}.

\medskip
\noindent\textbf{Acknowledgments}. Thanks to Toby Hall and
Jan Boronski for  useful conversations. I (re)discovered 
pinching using a program written by Toby. After this paper was written
I became aware of the paper \cite{paul}  in preprint form
which among its results
 on the PL case contain those 
in Theorem~\ref{another} here. The two 
papers inevitably share some techniques. 
 Thanks also to the referees whose suggestions improved
the exposition of this paper.

\section{preliminaries}

\subsection{spaces of maps on the circle}
The circle is $S^1 = \RR/\Z$.
The only measure used here is Lebesgue, and is denoted by
 $\lambda$. For a set, $X$, its indicator
function is denoted $I_X$ and if $X\subset Z$, then its
complement in $Z$ is $X^c = Z\setminus X$. For 
a homeomorphism $g:S^1\raw S^1$ the orbit of the point
$x$ is $o(x, g)$.

 The space of all real-valued $C^r$-functions $r\geq 0$
on $S^1$ is denoted $C^r(S^1)$ with the $C^r$-norm $\|\phi\|_{C^r}$,
and the space of all real-valued Lipschitz functions
is denoted $C^{Lip}(S^1)$ with the Lipschitz norm $\normL{\phi}$
 (defined below).
 These spaces  will often be treated
as functions $\phi:\RR \raw \RR$ with $\phi(x+1) = \phi(x)$.

We will be concerned here with degree one circle 
maps $f:S^1\raw S^1$. These maps have a lift $\tf:\RR\raw \RR$
which satisfy $\tf(x + 1) = \tf(x) + 1$. A map $f:S^1\raw S^1$
is a continuous degree one circle map if and only if it has a
 lift of the form 
$$
\tf(x) = x + \phi(x) 
$$
with $\phi\in C^0(S^1)$.

If $g$ is a degree-one nondecreasing circle map 
with a chosen lift $\tg:\RR\raw \RR$,
the rotation number of a point $\tx\in\RR$ is 
\begin{equation*}
\rho(\tg(\tx)) = \lim_{n\raw\infty} \frac{\tg^n(\tx) - \tx}{n}. 
\end{equation*}
The rotation number is independent of the choice of $\tx$
so one writes just $\rho(\tg)$.
\begin{convention}
The rotation number depends in a simple way on the choice
of the lift, but in this paper the lift will always be fixed and
specified so we just write $\rho(g)$.
With one important exception,  
$f$ will be used for both the circle map and its lift.  
The difference will be clear from the context. The exception
concerns the fundamental condition for pinching $\tf^q = R_p$.
In the circle the maps satisfy $f^q=\Id$, which loses information,
so it is important to distinguish the two. 
\end{convention} 

Both $C^r(S^1)$ and $C^{Lip}(S^1)$ are complete, and so closed subsets 
are Baire spaces.
If $\alpha\in L^1(\lambda)$ is essentially bounded, the infinity norm is
$\|\alpha \|_\infty = \esssup |\alpha| $. 
In finite dimensions, $\|\vv \|_\infty = 
\max |v_i|$.

\subsection{families, rotation numbers and pinching}
When $\phi\in C^{Lip}$ and standard-like as will be discussed in
Section~\ref{stdlike}  the corresponding standard-like family is 
\begin{equation}\label{sdef}
f_{b, \omega, \phi}(x) = x + \omega + b \, \phi(x)
\end{equation}
for $b\in [0,\infty), \omega\in S^1$.
The family~\ref{sdef} is studied
in the $(\omega, b)$ parameter plane using
the sets
\begin{equation*}
T_r = \{(\omega, b) : \rho(f_{b, \omega, \phi}) = r\}.
\end{equation*}
This set is called the $r$-tongue.
The basic properties of the $T_r$ are given in 
\cite{arnold, herman1, boyland, trouble}.

Each set $T_r$ is connected. 
For fixed $b$, the map $\omega\mapsto \rho(f_{b, \omega, \phi})$ is
continuous and nondecreasing. Thus each line of constant $b$ intersects
each $T_r$ in a (perhaps trivial) closed interval. When 
$r = \alpha$ an irrational, $T_r$ is a line segment
connecting the lines $b=0$ to $b=1$. When $r=p/q$
the set $T_{p/q}$ has a left boundary
characterized as all $(\omega, b)$ so that 
$\tf_{b, \omega, \phi}^q(x) \leq x + p$ with equality at some $x$. Similarly,
its right boundary is  all $(\omega, b)$ so that 
$\tf_{b, \omega, \phi}^q(x) \geq x + p$ with equality at some $x$. 
It follows that line of constant $b$ intersects $T_{p/q}$ in just one point
$(\omega, b)$ if and only if and $\tf_{b, \omega, \phi}^q  = R_p$.
This prompts the definition:
\begin{definition}
The forcing $\phi$ is called $p/q$-pinching if
for some $b\in (0, 1], \omega\in S^1$,
\begin{equation*}
\tf_{b,\omega,\phi}^q = R_p.
\end{equation*}
In this case, $(b,\omega, \phi)$ is called a $p/q$-pinch point.
\end{definition}
When we write $p/q$ it is always assumed that $p$ and $q$ 
are relatively prime and so $q>1$.
Note that when $b=0$, we obviously have $\tf_{0,p/q,\phi}^q = R_p$
for any $\phi$ and so we exclude $b=0$ from the definition
of pinching.
                        
The right and left boundaries of $T_{p/q}$ and $T_\alpha$ are
all graphs of Lipschitz functions with $\omega$ as a function of $b$.
The crucial observation for $C^r$-genericity 
is Herman's  (\cite{herman1}, III Proposition 3.2). 
\begin{lemma}[Herman]\label{herman}
If $\phi$ has an extension to a complex
entire function then it is not pinching for all $p/q$. 
This holds in particular when $\phi$ is a 
nonconstant trigonometric polynomial.
\end{lemma}

\subsection{Lipschitz functions}
Recall that if $\phi:\RR\raw\RR$ is Lipschitz, then $\phi'$ exists almost
everywhere, is in $L^1(\lambda)$, there is an $M$ with 
$ \esssup(|\phi'|)<M$,
 and for all $x, x_0$,
$$
\phi(x) =  \phi(x_0) + \int_{x_0}^x \phi' \; d\lambda.
$$
The usual norm on Lipschitz functions is 
$$
\normL{\phi} = \max|\phi(x)| + \sup_{y\not = x}|\frac{\phi(y) - \phi(x)}{y-x}|,
$$
but it is standard for a Lipschitz function $\phi$ that 
\begin{equation}\label{standard}
\inf_{y> x} \frac{\phi(y) - \phi(x)}{y-x} = \essinf \phi'(x)
\end{equation}
and so it is convenient to  use
$$
\normL{\phi} = \|\phi |_{C^0} + \|\phi'\|_\infty
$$ and so 
treat $\C^{Lip}(S^1)$ as the Sobolev space 
$ W^{1,\infty}(S^1)$. 

Given a function $W\in L^1(S^1)$ with $\esssup(|W|) < \infty$
and $\int W d\lambda = 0$ we will often
define a  function $G\in C^{Lip}(S^1)$ by
$$
G(x) = G(0) + \int_{0}^x W \; d\lambda.
$$
The Lebesgue Differentiation Theorem then yields $G' = W$ a.e.
In addition, if $\essinf W > -1$ then
$$
g(x) = x + g(x_0) +  \int_{x_0}^x W \; d\lambda
$$
is the lift of an invertible degree one circle
map with Lipschitz inverse $g^{-1}$. Letting $X$ be
the full measure set where $g'$ exists, since Lipschitz functions 
are absolutely continuous, $g^{-j}(X)$ is also full measure for any $j$
 and thus for $q > 0$,
\begin{equation}\label{goodset}
X^{(q)} = X \cap g^{-1}(X) \dots \cap g^{-q+1}(X)
\end{equation}
is also full measure. Thus for $x\in X^{(q)}$ the chain rule holds, so
$$
(g^q)'(x) = g'( g^{q-1}(x))\; g'( g^{q-2}(x)) \dots g'(x).
$$
The set $X^{(q)}$ in \eqref{goodset} is called the \textit{good set}
for $g$ and $q$.

 For brevity throughout the paper when we speak of  the derivative
of a Lipschitz function it always means the almost everywhere derivative.

\section{standard-like families}\label{stdlike}  As noted in the introduction
we consider families with the same structure as the standard
family. A map is strongly order preserving if $x_1 < x_2$ implies
$f(x_1) < f(x_2)$ while it is weakly order preserving if
 $x_1 < x_2$ implies $f(x_1)\leq f(x_2)$.
\begin{definition}
$\phi$ is called standard-like if
\begin{align*}
f_{b,\omega,\phi}\ \text{is strongly order preserving for}\ b<1,\\
f_{b,\omega,\phi}\ \text{is not weakly order preserving for}\ b>1.\\
\end{align*}
\end{definition}

When $b<1$ $f_{b,\omega,\phi}$ is a homeomorphism. When
$b=1$ it can be a homeomorpism (as in the standard family)
or noninjective and weakly order preserving (as in the
triangle wave example). The condition for $b>1$ says that
there are $x_1<x_2$ but $f(x_1)> f(x_2)$ which is stronger
than just being noninjective. Thus, when $b>1$,
 $\tf_{b,\omega,\phi}^q$ is never $R_p$ and so we exclude 
$b>1$ from any consideration of pinching.

The collection of standard-like $\phi\in C^r(S^1)$ is denoted
$\mathcal{S}^r$ while those in $C^{Lip}(S^1)$ is denoted
$\cSL$.
In the smooth and Lipschitz case
the standard-like $\phi$ are easily characterized.
\begin{lemma}\label{char}
The smooth ($r>0$) standard-like $\phi$ are
\begin{equation*}
\mathcal{S}^r = \{ \phi\in C^r(S^1)\colon \min \phi' = -1\},
\end{equation*}
and the Lipschitz  standard-like $\phi$ are
\begin{equation*} 
\cSL = \{ \phi\in \Lip(S^1) \colon 
\essinf \phi' = -1\}
\end{equation*}
\end{lemma}

\begin{proof}
We prove the Lipschitz case $\phi\in C^{Lip}(S^1)$;  
the smooth case then follows.
Recall that given $\phi$ the corresponding standard-like family
is  $f_{b,\omega,\phi}$ as defined in \eqref{sdef}.
Each element of the family 
is the sum of Lipschitz functions and so each $f_{b,\omega,\phi}$
is Lipschitz 
  and thus $f_{b,\omega,\phi}^\prime = 1 + b\phi'$ exists a.e. 

Now assume that $\essinf \phi' = -1$ and so using \eqref{standard},
$\inf_{y\not = x}\frac{f_{b,\omega,\phi}(y) - f_{b,\omega,\phi}(x)}{y-x}=
1 - b$.
Thus if $b<1$, $f_{b,\omega,\phi}$ is a strongly order preserving,
and if $b>1$,  $f_{b,\omega,\phi}$ is not weakly order preserving. Thus 
$\phi\in \mathcal{S}^{Lip}$.

Next assume that $\essinf \phi' \not= -1$ and note
that $\essinf f_{b,\omega,\phi}' = 1 + b \essinf \phi'$.
If $\essinf \phi' < -1$, then pick $b$ so that $1 > b > -1/\essinf \phi'$
and so $b \essinf \phi < -1$ thus  
$\essinf f_{b,\omega,\phi}' < 0$ and so is not order preserving
for some $b<1$ and
$\phi\not\in \mathcal{S}^{Lip}$.

Now if $\phi$ is constant, clearly $\phi\not \in \mathcal{S}^r$ and
so since $\phi$ is periodic,   $\essinf \phi' < 0$.   
 If $0 > \essinf \phi'  > -1$  and pick $b$ 
with $1 < b < -1/\essinf \phi' $ so $b \essinf \phi' > -1$.
Thus   $\essinf f_{b,\omega,\phi}^\prime > 0$ and so $f_{b,\omega,\phi}$ is
strongly order preserving
 for a $b>1$ and so again $\phi\not\in \mathcal{S}^{Lip}$.
\end{proof}

\begin{remark} Given a homeomorphism
$g:S^1\raw S^1$ which is in one of the classes of smooth, Lipschitz
or PL,
  let $\hphi(x) = g(x) - x$ and then 
$\hphi(x+1) = \hphi(x)$ and  $\phi = \hphi/|\essinf \hphi'|$
is standard-like. Thus   
$$ g(x) = x + |\essinf \hphi'| \phi,
$$
and so any such $g$ is included in a standard-like family with
a forcing in the same class as $g$. Also note that if 
$g = h \circ R_{p/q} \circ h^{-1}$ with $h'(x + p/q) \not= h'(x)$
for some $x$, then $\tg^q = R_p$ and $g\not = R_{p/q}$. Thus $g$ is a member of 
a standard-like family with $b>0$ and is a pinch point. Thus there
are pinch points for all considered classes of forcing. 
\end{remark}
\section{$C^r$ and Lipschitz genericity theorem}\label{main}
The proof of the  theorem below will be contingent on the density
results proved in the next two sections. The PL genericity
theorem needs somewhat different methods and is proved in
 Section~\ref{notlast}. 
\begin{theorem}\label{gen1}
The collection of all 
$\phi\in\mathcal{S}^r$, $r>0$ which  are not pinching for all $p/q$ 
is dense, $G_\delta$ (generic) as are those
 $\phi\in\mathcal{S}^{Lip}$.
\end{theorem}
\begin{proof}
For the  $C^r$ case, first note that using Lemma~\ref{char}
it is easy to check that $\mathcal{S}^r$  is closed
in the complete metric space $C^r(S^1)$  and is thus a Baire space.

For a fixed $p/q$ and $n>1$ let
\begin{equation*}
A_{p/q, n} =
 \{ \phi\in\mathcal{S}^r\colon \tf_{b,\omega,\phi}^q \not= R_p \ \text{for all}
\ b\in [1/n, 1],\omega\}.
\end{equation*}
We first show $A_{p/q,n}$ is open in $\mathcal{S}^r$.
 The complement
of $A_{p/q,n}$  in $\mathcal{S}^r$ is 
\begin{equation*}
B_{p/q, n} = \{ \phi\in\mathcal{S}^r\colon \tf_{b,\omega,\phi}^q = R_p \ \text{for some}
\ b\in [1/n, 1],\omega\}.
\end{equation*}
So assume that $\phi_k\raw\phi_0\in\mathcal{S}^r$
and for each $k$, $\tf_{b_k,\omega_k,\phi_k}^q = R_p$. Passing to subsequences,
there are $b_0$ and $\omega_0$ and $b_{k_i}\raw b_0$ and
 $\omega_{k_i}\raw \omega_0$
and since $\phi_{k_i}\raw \phi_0$ in $C^0(S^1)$,  
$\tf_{b_0,\omega_0,\phi_0}^q = R_p$ and so 
$\phi_0\in B_{p/q,n}$ and $B_{p/q,n}$
is closed, and so $A_{p/q,n}$ is open.

Finally, by Lemma~\ref{smooth} below,  $A_{p/q,n}$ is also  
dense in $\mathcal{S}^r$ and thus  the collection of
 not pinching forcings
\begin{equation*}
A = \bigcap_{p/q,n} A_{p/q,n} =
\{ \phi\in\mathcal{S}^r\colon \tf_{b,\omega,\phi}^q \not= R_p \ \text{for all}
\ b>0,\omega, p/q\}
\end{equation*}
is dense $G_\delta$.

The proof for the Lipschitz case is almost identical, except
it finishes with Theorem~\ref{lipschitz}.
\end{proof}

\begin{remark}\label{czero}
 Note that no $C^0$ version of this theorem is
given. A proof like that of Lemma~\ref{char} shows that 
\begin{equation*}
\mathcal{S}^0 = \{ \phi\in C^0(S^1)\colon \inf_{y> x} 
\frac{\phi(y) - \phi(x)}{y-x}= -1\}.
\end{equation*}
In addition, trigonometric polynomials are dense in $C^0(S^1)$
and are not pinching by Lemma~\ref{herman} and so a proof like
that of Lemma~\ref{smooth} shows that not pinching
is dense in $\mathcal{S}^0$. The difficulty is
that $\mathcal{S}^0$ is neither open nor closed in $C^0(S^1)$
and so it is not clear that it is a Baire space.
\end{remark}

\section{The smooth density lemma}

\begin{lemma}\label{smooth}
For each $p/q, n>0$ the set 
\begin{equation*}
A_{p/q, n}^r =
 \{ \phi\in\mathcal{S}^r\colon \tf_{b,\omega,\phi}^q \not= R_p \ \text{for all}
\ b\in [1/n, 1],\omega\}.
\end{equation*}
is dense in $\mathcal{S}^r$.
\end{lemma}

\begin{proof}
It is classical that trigonometric polynomials are dense in $C^r(S^1)$ 
for all $0 \leq r \leq \infty$.  We then rescale them to get
elements in $\cSr$. They are not pinching by
Lemma~\ref{herman}. So the main task it to show 
 that the rescaling
can be done while maintaining their density.

Given $\phi\in\mathcal{S}^r$ and $\epsilon>0$, let $P$ be a nonconstant
 trigonometric polynomial with  
$$\| \phi-P\|_{C^r} < \delta:= \min(\frac{\epsilon}{4 \|\phi\|_{C^r}},
 \frac{\epsilon}{4})
$$ Letting $C = \min(P')$ then $C<0$ and
$$
|1 - |C|| = |C+1| = |\min P' - \min \phi'| < \delta.                     
$$
Thus since $|C| > 1/2$,
$$
\| \frac{P}{|C|} - P\|_{C^r} = |\frac{1}{|C|} -1|\|P\|_{C^r} 
\leq |\frac{1-|C|}{|C|}|(\|\phi\|_{C^r} +\delta) < 
2 \delta(\|\phi\|_{C^r} + \delta)
$$
and so
$$
\| \frac{P}{|C|} - \phi \|_{C^r}\leq 
\|\phi - P\|_{C^r} + \|P - P/|C|\|_{C^r} \leq 
\delta+  2 \delta(\|\phi\|_{C^r} + \delta)  < \epsilon
$$ and
$P/|C|$ is a trigonometric polynomial in $\mathcal{S}^r$ and 
 it is in $A_{p/q,n}$.
\end{proof}

\section{the Lipschitz density theorem}

\begin{theorem}\label{lipschitz}
For each $p/q, n>1$ the set 
\begin{equation}\label{Adef}
A_{p/q, n}^{Lip} =
 \{ \phi\in\cSL\colon \tf_{b,\omega,\phi}^q \not= R_p \ \text{for all}
\ b\in [1/n, 1],\omega\}.
\end{equation}
is dense in $\cSL$.
\end{theorem}
Throughout this section the given $\phi\in\cSL$ is fixed as
are $p/q$ and $n$.

 To prove Theorem~\ref{lipschitz},
 given $\phi\in\cSL$ and $\epsilon > 0$ we must produce
a $\psi\in A_{p/q, n}^{Lip}$ with $\normL{\phi - \psi} < \epsilon$.
The proof  requires several
lemmas and is spread over several subsections.

The first step   (subsection~\ref{Aone}) makes
this approximation problem finite dimensional. It
applies to the derivative of $\phi$  
a variant of the standard proof that simple functions are dense in $L^\infty$.
 Given $\phi$, it
 produces an $N$-element partition $\{X_j\}$
and a system of weights $\hbw \in\RR^N$ so that the Lipschitz function
$\hpsi$ with $\hpsi(0) = \phi(0)$ and derivative 
$\hpsi' = \sum \hw_j I_{X_j}$ has the properties that
 $\normL{\phi - \hpsi}$ is small and $\hpsi \in \cSL$. 
We now fix the partition $\{X_j\}$. If $\hpsi$ is not pinching
we are done, so assume it is. The
rest of the proof is devoted to finding
a new system of weights $\bw\in\RR^N$
so that the $\psi_{\bw} \in \cSL$ is close to $\hpsi$, has  derivative 
$\psi_{\bw}' = \sum w_j I_{X_j}$,  and
crucially, $\psi_{\bw}$ is not pinching. 
 
Thus the next step (subsection~\ref{plaus1}) is to find conditions which must
occur when $\psi_{\bw}$ is pinching, or when there
exist  $b$ and $\omega$ so that
\begin{equation}\label{mainpinch}
\tf_{b,\omega, \psi_\bw}^q = R_p.
\end{equation}
Because of their simple form, it easiest to work with derivatives. 
When $x \in X_j$ then $ (f_{b,\omega, \psi_\bw})'(x) = 1 + b w_j$ 
and so \eqref{mainpinch}
implies that
\begin{equation}\label{mainpoly}
1 = (f_{b,\omega, \psi_\bw}^q)'(x) = \prod_{j=0}^{q-1} (1 + b w_j)
\end{equation} 
where $f_{b,\omega, \psi_\bw}^j(x) \in X_j$
 (note the absence of $\omega$). By choosing
the number of partition elements larger than $3$ times the period
(or $N>3q$) we can ensure that when $\psi_{\bw}$ is pinching
 various subsets of the weights $\bw$ satisfy three distinct polynomials
of the form \eqref{mainpoly}, all for the same value of $b$. 

The last part of the proof (subsections~\ref{Athree}, \ref{Afour}, 
\ref{Afive}) uses a geometric
argument to show that near the given $\bw$ there is an
open dense set of weights with the property that
they do not solve $3$ distinct polynomials of the form
\eqref{mainpoly} all with the same value of $b$. Thus
 we may perturb $\bw$ to a new set
of weights $\bu$ so that   $\psi_{\bu}$
is not pinching.

\subsection{ Discretization}\label{Aone}
 For an interval $J$, $|J|$ is
its length or equivalently, its Lebesgue measure.
\begin{lemma}\label{one}
Given  $\epsilon > 0$, $\phi\in\cSL$, and $p/q$ there exists a 
 finite measurable partition $\cP = \{X_1, \dots, X_N\}$ of $S^1$ with
$N> 3q$ and each $\lambda(X_j)>0$  and weights
 $\hbw = (\hw_1, \dots, \hw_N)$ so that
$$
\hpsi(x) := \phi(0) + \int_0^x \sum_{j=0}^N \hw_j I_{X_j}\; d\lambda
$$
is in $\cSL$ and $\normL{\phi - \hpsi} < \epsilon$.
\end{lemma}

\begin{proof}
By the characterization of $\cSL$,  $ \essinf \phi' = -1$ and 
$C:= \esssup \phi' < \infty$. Find half open intervals 
$K_i = [a_i, b_i)$
for $i=1, \dots N-1$ with $K_1  = [-1,  b_1)$
 and a closed interval $K_N = [a_N, C]$ so that
\begin{enumerate}
\item If $X_i  = \{ x\in S^1: \phi'(x)\in  K_i\}$ then
$\lambda(X_i) > 0$ and $\sum \lambda(X_i) = 1$.
\item $|K_i| < \frac{\epsilon}{2}$.
\item $N>3q$.
\end{enumerate}

Let $m_i = \lambda(X_i)$ and for $i = 1, \dots, N$ let 
$$
w_i = \frac{1}{m_i}\int_{X_i} \phi'\;  d\lambda
$$
and  $W = \sum w_i  I_{X_i}$. Thus 
$$
\psi(x) = \phi(0)   + \int_0^x W \; d\lambda.
$$
has $\hpsi' = W$ a.e. By construction, $\sum m_i w_i = 0$
and so $\psi$ is periodic. Finally, we show that 
\begin{equation*}
\normL{\phi - \psi} \leq \epsilon.
\end{equation*}
 Since
$w_i\in K_i$, $\|\phi'-\psi'\|_\infty < \frac{\epsilon}{2}$.
And for the  $C^0$ part of the Lipschitz norm
\begin{align*}
& \sup|\phi(x) - \hpsi(x)|  \\
&= \sup\left|\phi(0) + \int_0^x \phi' \; d\lambda -
 \left(\phi(0) + \int_0^x W \; d\lambda\right)\right| \\
&\leq \sup\int_0^x |\phi' - W| \; d\lambda\\
&\leq  \int_0^1 |\phi' - W| \; d\lambda \leq 
\frac{\epsilon}{2}.
\end{align*}

The last step in getting a $\hpsi\in\cSL$ is to rescale $\psi$ so that
$\essinf \hpsi' = -1$. 
Let $C= |w_1|$ and $\hpsi = \psi/C$  so $\hpsi$ has
weights $\hw_i = w_i/C$. By an argument just like 
Lemma~\ref{smooth} by choosing $\normL{\phi - \psi}$ small enough we
can make  $\normL{\phi - \hpsi}$ arbitrarily small.
Since $\hpsi = \sum \frac{w_j}{C} I_{X_j}$ and  
$ \sum (w_i/C m_i) = 0$,  $\hpsi$ is periodic and
$\hpsi\in\cSL$.
\end{proof}

Now given $\hpsi$ we generate a family 
$$
f_{b,\omega, \hpsi}(x)  = x + \omega + b\, \hpsi(x).
$$
If $\hpsi$ is not pinching, we have proved Theorem~\ref{lipschitz}.
 If it is pinching, we fix the partition $\cP$
 from Lemma~\ref{one} and perturb the vector of weights $\bw$
to obtain a  new $\psi$ that is not pinching.
 
Given a vector of masses $\bm$ there are two conditions required
 for a vector of weights $\bw$
to generate a $\psi_\bw\in\cSL$:
\begin{equation}\label{cond}
\sum m_j w_j = 0 \ \text{and}\ w_1 = -1.
\end{equation}
The first guarantees that the resulting $\psi_\bw$ 
is periodic and the  second condition gives  $\psi_\bw\in\cSL$. 
If these two conditions are satisfied, define
$$
\psi_\bw(x) := \phi(0) + \int_0^x \sum w_j I_{X_j}\; d\lambda.
$$
Next we must  check that a small perturbation of the
weights yields a small perturbation of the resulting $\psi$.
Given two weight vectors $\bu$ and $\bw$ first note that  
$$
\|\psi'_{\bu} - \psi_{\bw}'\|_\infty =
\|\sum (u_i-w_i) I_{X_i}\|_\infty =
 \|\bu -\bw\|_\infty.
$$
To check the $C^0$ part of the Lipschitz norm 
\begin{align*}
 \sup\left|\phi(0) + \int_0^x \sum u_i I_{X_i} \; d\lambda\right. 
& - \left. \left(\phi(0) + \int_0^x \sum w_i I_{X_i}\right) \; d\lambda
\right| \\
&\leq \sup\int_0^x |\sum(u_i-w_i)| I_{X_i} \; d\lambda\\
&\leq  \int_0^1 \sum |u_i-w_i| I_{X_i} \; d\lambda \\
&\leq 
\| \bu - \bw\|_\infty.
\end{align*}

Thus
\begin{equation}\label{twolip}
\normL{\psi_{\bu} - \psi_{\bw}} \leq 2  \|\bu - \bw\|_\infty
\end{equation}
which is to say that the map
$\bw\mapsto \psi_\bw$ is $2$-Lipschitz with these norms.
Thus we adopt the strategy of perturbing  $\hbw$
to a not pinching $\bw$ satisfying \ref{cond}.

\subsection{plausible polynomials}\label{plaus1}
 The simple form of $f_{b,\omega, \psi_\bw}$ makes
the derivatives along orbits   tractable
to work with. Thus 
we first explore the consequences of $(\tf_{b,\omega, \psi_\bw})^q = R_p$
for the derivatives and see that 
$b$ and $\bw$ must solve a system of equations. We then 
find a perturbation $\bw$ for which no such equations can be solved.

Let $\Lambda_N = \{1, \dots, N\}^q/S_N$  where the symmetric group
$S_N$ acts in the standard way on indices. Thus $J\in\Lambda_N$ can
be described by the elements $j$ present  and their multiplicity $e_j$.
 Given $\bw\in \RR^N$, $b\in\RR$ and  $J\in\Lambda_N$ define 
\begin{equation}\label{formA}
G(b, \bw, J) = \prod_{j\in J} (1 + b w_j).
\end{equation}
The definition of $\Lambda_N$ makes this well-defined.
\begin{definition} Given a vector of weights $\bw$, an index set 
 $J$ is called  $(1/n)$-plausible if the equation 
$
G(b,\bw, J) = 1
$
has a solution $b\in  [1/n, 1]$. The associated polynomial
$p(y) = \prod_{j\in J} (1 + y w_j)$ is also called 
$(1/n)$-plausible.
\end{definition}
We will need a lemma from Calculus that determines which
polynomials of the form \eqref{formA} are $(1/n)$-plausible for
some $n$.
\begin{lemma}\label{calculus}
Let 
$$
p(x) = \prod_{i=1}^m (1 + y k_i)
$$
with $-1 \leq k_1 \leq \dots \leq k_m$,  and $k_1 < 0 < k_m$. 
Then $p(y) = 1$ has a solution in $(0,1)$  if and only
if $\sum k_i > 0$ and $\prod (1 + k_i) < 1$. Further, if there
is a solution $y^*$ in $(0,1)$, it is unique and $p'(y^*) < 0$.
\end{lemma}

\begin{proof}
Note first that $p(0) = 1$, $p(1) = \prod (1+k_i)$, and $p'(0) = \sum k_i$.
Further, since all of $p$'s roots (i.e. solutions to $p(y) = 0$)
are real, between each adjacent
pair of roots (which may have multiplicity) there is a unique
critical point which is either the maximum or minimum on
the interval between the roots.
There are a pair of adjacent roots  for $p$,
$x_L = -1/\max k_i<0$ and 
$x_R = -1/\min k_i \geq 1$ and so there is a critical point $x_c\in [x_L, x_R]$.
Since $p(0) = 1 > 0$ it is a maximum.

Assume first that $p'(0) = \sum k_i > 0$ and $ p(1) = \prod 1 + k_i < 1$.
Since now $p(1) < p(0)$, the maximum $x_c\in (0,1)$. There then
exists a unique $y^*\in (0,1)$ with $p(y) = 1$, and since $p'(0) > 0$,
$p'(y^*) < 0$. 

For the converse, say $p'(0) = \sum k_i \leq 0$. Since there
at most  one critical point in $(0,1]$ and it is a maximum, 
$p$ is decreasing from $p(0) = 1$ on that interval and thus
$p(y) < 1$  in $(0,1]$. 
Now say $p(1) = \prod (1 + k_i) > 1 = p(0)$, then again since
there is at most one critical point, a maximum, in $(0,1]$,
$p(x) > 1$ in $(0,1]$.
\end{proof}
In light of Lemma~\ref{calculus}, if $p(y)$ is $(1/n)$-plausible,
the unique $y\in [1/n, 1]$ with $p(y) = 1$ is
called the \textit{associated solution}.

The next lemma connects $(1/n)$-plausible polynomials
 to  derivatives at pinch points.
\begin{lemma}\label{pincheq}
Assume $(b,\omega, \psi_\bw)$ is a $p/q$-pinch point with $b\geq 1/n$,
then there are least three different $(1/n)$-plausible index sets for  $\bw$
such that all their associated polynomials
$p$ have the same solution to $p(y) = 1$, namely,  $y=b$.
\end{lemma} 
\begin{proof}
Assume $x$ is in the good set for a $f_{b,\omega, \psi_{\bw}}$ 
as defined in \eqref{goodset}.
Since $f_{b,\omega, \psi_{\bw}}'(x) = 1 + b \psi_{\bw}'(x) = 1 + b w_i$
 when $x\in X_i$,   it follows that 
\begin{equation}\label{save}
1 = (f_{b,\omega, \psi_{\bw}}^q)'(x) = 
(1 + b w_{i_0}) \dots  (1 + b w_{i_{q-1}})
\end{equation}
when   $f_{b,\omega, \psi_{\bw}}^j(x) \in X_{i_j}$.
Since $\lambda(X_i)>0$ for all $i$ and $N>3 q$, as we vary $x$  there
must be at least $3$ such distinct equations all solved
by the same $b$ and each $w_i$ must occur
in at least one of them. 
\end{proof}

We allow the solution $b=1$ in the definition of a $(1/n)$-plausible
index set, but it will never occur in a equation like \eqref{save}
for a $\tf_{b,\omega, \psi_{\bw}}^q = R_p$ with $\psi_{\bw} \in
\mathcal{S}^{Lip}$. The  condition $\psi_{\bw} \in
\mathcal{S}^{Lip}$ forces $w_1 = -1$ (recall
\eqref{cond}) which means that $\psi_{\bw}' = -1$ on
 $X_1$ and thus $f_{1,\omega, \psi_{\bw}}'
= 0$ on the positive measure set
$X_1$ and so $\tf_{1,\omega, \psi_{\bw}}^q \not= R_p$.

\subsection{The plausible function}\label{Athree}
Continuing under the assumption that $(\hb,\homega, \psi_\hbw)$
is a $p/q$-pinch point
 for some
$\hb$ and $\homega$, Lemma~\ref{pincheq}
 implies that for  $n>1/b$ 
the set of $(1/n)$-plausible
index sets $J$ for $\hbw$ is nonempty.
Let $J_i$ for $i = 1, \dots, M$ be the list of $(1/n)$-plausible index sets
for $\hbw$. For each $i$ define $F_i:\RR \times \RR^N \raw \RR$ as 
$
F_i(b, \bw) = G(\bw, b, J_i).
$
We treat each of these $b$'s as a separate variable $b_i$ and  let
$\bb = (b_i) \in \RR^M$  and collect the functions
together into $F = (F_i)$ and so 
$F:\RR^M \times \RR^N\raw \RR^M$. 
Letting $\One\in\RR^M$ be 
$\One = (1, 1, \dots, 1)$,  since we have restricted to $(1/n)$-plausible
$J$ using Lemma~\ref{calculus} we have that the equation
$
F(\bb, \hbw) = \One
$
has a unique solution with all $b_i\in[1/n, 1]$. 
 We call this solution $\hbb$ and so 
$
F(\hbb, \hbw) = \One.
$

The goal is to perturb $\bw$ and see its effect on $\bb$.  
We use the Implicit Function Theorem to show that for $(\bb,\bw)$
near $(\hbb, \hbw)$ the set of the solutions to $F(\bb, \bw) = \One$
can be described as $F(g(\bw), \bw) = \One$ with $g$ smooth. Thus
we can perturb $\bb$ in a controlled manner by perturbing $\bw$.

Using Lemma~\ref{calculus}, 
$$
\alpha_i := \frac{\partial F_i}{\partial b_i}(\hbb, \hbw) < 0 \ \text{and} \
\frac{\partial F_i}{\partial b_j}(\hbb, \hbw) = 0  \ \text{when} \ j\not = i
$$
and so 
$\frac{\partial F}{\partial b} (\hbb, \hbw) = 
\diag( \alpha_1,  \dots, \alpha_M)$ which is invertible.
Thus there exists $U$ an open neighborhood of $\hbw$  and a 
smooth function $g:U\raw \RR^n$ so that for $\bw\in U$,
$F(\bb, \bw) = \One$ if and only if $\bb = g(\bw)$. Again
by the Implicit Function Theorem
\begin{equation*}
A := \frac{\partial g}{\partial \bw}(\hbb, \hbw) = 
-\left(\frac{\partial F}{\partial \bb} (\hbb, \hbw)\right)^{-1}
\frac{\partial F}{\partial \bw} (\hbb, \hbw).
\end{equation*}

To compute the derivative of $F$ with respect to $\bw$
we use an alternative form of the function. 
 Recall that $J\in\Lambda_N$ can
be described by the elements $j$ present  and their multiplicity $e_j$. For
a given $J_i$, let $\overline{J_i}$ be the collection of distinct
elements $w_j$ in $J_i$ and $e_{i,j}$ be their multiplicity. So 
if we have representative $J_i = (2, 2, 3, 4, 4, 4)$,
then $\overline{J}_i = (2, 3, 4)$ with multiplicities
$e_{i,1} = 2$, $e_{i, 2} = 1$, and $e_{i,3} = 3$.
Thus we are writing 
$$
F_i(\bb, \bw) = \prod_{j\in \overline{J}_i} (1 + b_i w_j)^{e_{i,j}}
$$
where the  $w_j$ term occurs just once in the product.

Now  $F_i(\hb_i, \hbw) = 1$  and 
when $j\in J_i$,
$$
\frac{\partial F_i}{\partial w_j}(\hbb, \hbw) 
= \hb_i e_{i,j} (1 + \hb_i \hw_j)^{e_{i,j}-1}
 \prod_{k\in \overline{J}_i, k\not= j} (1 + \hb_i \hw_k)^{e_{i,k}}
=\frac{\hb_i e_{i,j}}{1 + \hb_i \hw_j}
$$
and this is zero when $j\not\in J_i$.

Thus, $A_{i,j} = 0$ when $j\not\in J_i$ and when $j\in J_i$
\begin{equation}\label{diff}
A_{i,j} = \frac{- \hb_i e_{i,j}}{\alpha_i(1 + \hb_i \hw_j)}.
\end{equation}

While it is admittedly somewhat pedantic it will avoid a certain
amount of confusion  to distinguish the tangent spaces $T_{\hb}$
and $T_{\hbw}$ at $\hbb$ and $\hbw$
from the base spaces. We use primes to indicate vectors and points in
the tangent space. Thus $\bb = \bb' + \hbb$ and $\bw = \bw' + \hbw$. 
The exponential maps will both be denoted $\pi$ so 
$\pi(\bb') = \bb$ and $\pi(\bw') = \bw$. As computed above
the tangent map induced by $g$ is $A:T_{\hbw}\raw T_{\hbb}$.

\subsection{the linear perturbation lemma}\label{Afour}
Recall the goal is to perturb $\psi_\hbw$ to obtain a not pinching
$\psi_\bw$. Using Lemma~\ref{pincheq} this could 
be accomplished with a $\bw$ so that  $\bb = g(\bw)$ has
the property that no two of its coordinates are equal. This corresponds
to $g(\bw)$ avoiding the subspaces where two coordinates are equal.
So for $i_1 \not = i_2$ define the subspace of $\RR^M$ by
$V_{i_1, i_2} = \{ \bb : b_{i_1} = b_{i_2}\}$.
In the tangent space $T_{\hbb}$ this corresponds to 
$$V_{i_1, i_2}^\prime = 
\{ \bb' : b_{i_1}' + \hb_{i_1} = b_{i_2}'  + \hb_{i_2}\}$$

For each $b\in [1/n, 1]$ let $\Lambda(b) = \{i : \hb_i = b \}$. If for
all $b$ the set $\Lambda(b)$ consists of just one or no elements,
we are done. Let $\{b^{(1)}, \dots, b^{(k)}\}$ be the distinct
$b$'s for which $\Lambda(b)$ has more than one element. 
For these $b$  and $\ell = 1,, \dots, k$ let
$$
\Omega(b^{\ell}) = (\Lambda(b^{(\ell)})) \times \Lambda(b^{(\ell)}) \setminus
\Delta
$$
where $\Delta$ is the diagonal of $\{1, \dots, M\}^2$. 
Finally, let $\Omega = \cup_\ell \Omega(b^{\ell})$, so $\Omega$ is
all the indices $(i_1, i_2)$ with $i_1 \not = i_2$ and
 $b_{i_1} = b_{i_2}$.

It follows that  $V_{i_1, i_2}'$ is subspace of the tangent space  $T_{\hb}$
when $(i_1, i_2) \in \Omega$ 
and is otherwise an affine subspace of $T_{\hb}$ that avoids the origin.
Going back to the Implicit Function Theorem we may thus shrink the
neighborhood $U$ to $U'$ so that $A(U') \cap V_{i_1, i_2}' = \emptyset$ for
all $(i_1, i_2) \not\in \Omega$. In $T_{\hbw}$ let 
 $\bm^\perp = \{\bw' : \bw' \cdot \bm\} = 0$, 
$W_0 =  \{\bw' : w'_1= 0\}$,  and $U^*  = U' \cap \bm^\perp \cap W_0$.
Thus $\bw\in U^*$ will satisfy the conditions \eqref{cond} since
$w_1' = 0$ corresponds to $w_1 = w_1' + \hw_1 = -1$ and $\hbw' \cdot\bm
= (\bw - \hbw)\cdot \bm = \bw\cdot\bm$ since $\hbw\cdot\bm = 0$.

A subset $K\subset \RR^N$ is called a \textit{cone} (more
properly a full cone over the origin) if
whenever a nonzero vector $\bv\in K$ then $t \bv\in K$ for
all $t\not = 0$. When $V\subset \RR^N$ is a subspace of nonzero
codimension, $V^c$ is a cone.  

We first show the existence of the appropriate perturbation 
in the tangent plane $T_{\hbw}$  and then in the next section use
 Taylor's Theorem to get the perturbation required for Theorem~\ref{lipschitz}.
\begin{lemma} \label{lemma2}
For each $(i_1, i_2) \in \Omega$ there is
a vector $\vv' \in U^* = U' \cap \bm^\perp \cap W_0$ as defined
above so that 
$A \vv' \not\in V_{i_1, i_2}'$. 
Thus  
\begin{equation}\label{claim2}
\Gamma= \{ \bw'\in U^* : A \bw' \cap V_{i_1, i_2}' = \emptyset\ 
\text{for all}\ (i_1, i_2)\}
\end{equation}
\end{lemma}
is open and dense and there exists a $c>0$ so that
for all nonzero $\bw'\in \Gamma$, $t\bw'\in \Gamma$ for all 
nonzero $|t|< c$.

\begin{proof}
For $(i_1, i_2) \in \Omega$ consider the $4 \times N$ matrix
\begin{equation*}
B
= \begin{pmatrix}
r_{i_1}(A)\\
r_{i_2}(A)\\
\bm\\
E_1
\end{pmatrix}
\end{equation*}
where $r_i(A)$ is the $i^{th}$ row of the matrix $A$ and
$E_1 = (1, 0, 0, \dots , 0)$.  We will show that $B$ 
has rank $4$ and so the equation
\begin{equation}
  B \vv = \begin{pmatrix} b' \\ b'' \\0 \\ 0 \end{pmatrix}
\end{equation}
has a solution $\vv$ for arbitrarily small $b'\not= b''$ which
will prove the first sentence of the lemma.

Assume then that 
\begin{equation}\label{lincom}
\beta_1 r_{i_1}(A) + \beta_2 r_{i_2}(A)
+ \beta_3 \bm + \beta_4 E_1 = 0.
\end{equation}
  First, note that since
$N > 3q$ there are always $j>1$ with $A_{i_1 j} = A_{i_2 j} = 0$.
Since all $m_j>0$ this implies that $\beta_3 = 0$. 
Second we show   that a special case can't happen. Consider the 
pair of equations for $j>1$ with $a\not = c$ 
\begin{equation}\label{noway1}
\begin{aligned}
(1 + b \hw_1)^a (1 + b \hw_j)^{q-a} &= 1\\
(1 + b \hw_1)^c (1 + b \hw_j)^{q-c} &= 1.
\end{aligned}
\end{equation}
Recalling that $\hw_1 = -1$ solving, yields
$(1-b)^{a-c} = (1 + b\hw_j)^{a-c}$. One case is that 
$1-b = 1 + b \hw_j$ so $b=0$ or $\hw_j = -1$. The other is     
$b-1 = 1 + b \hw_j$ so $-2 = b(\hw_j-1) > -2 b$ so
$b >1$. The conclusions $b=0$ or $\hw_j = -1$ and 
$b >1$  cannot occur under our assumptions, and
thus the situation in \eqref{noway1} cannot occur for $\hbw$.

Now there are two cases. Assume there is a $j>1$ with $A_{i_1 j} = 0$ 
and $ A_{i_2 j} \not= 0$ or vice versa. In this case
$\beta_1 = \beta_2 = 0$ and so $\beta_4 = 0$ as well.
So assume now that no such $j$ exists so 
we are left with the case
that the equations for $i_1$ and $i_2$ contain the same $w_j$ with
perhaps  an exception at $j=1$. Since the two equations must
be different but the sum of the exponents of both must be $q$
there are $k,\ell > 1$ with  $k \not = \ell$ with
$e_{i_1, k} > e_{i_2, k}$ and $e_{i_1, \ell} \leq e_{i_2, \ell}$.
Now from \eqref{lincom} $\beta_1 A_{i_i k} = - \beta_2 A_{i_2 k}$
and   $\beta_1 A_{i_i \ell} = - \beta_2 A_{i_2 \ell}$. Thus
letting $b = b_{i_1} = b_{i_2}$ and for $j = 1,2$
$z_{j} = \frac{\partial F_{i_j}}{\partial b}(\hbw, \hb)$ 
using \eqref{diff},
\begin{align}
\frac{\beta_1 b e_{i_1,k}}{z_1(1 + b\hw_k)}
&= - \frac{\beta_2 b e_{i_2, k}}{z_2(1 + b\hw_k)}\\
 \frac{\beta_1 b e_{i_1,\ell}}{z_1(1 + b\hw_\ell)}
&= -\frac{\beta_2 b e_{i_2,\ell}}{z_2(1 + b\hw_\ell)}
\end{align}
and then solving we get the contradiction
$e_{i_1, k} / e_{i_2,k} = e_{i_1,\ell}/ e_{i_2,\ell}$
which implies 
$\beta_1 = \beta_2 = 0$ and so $\beta_4 = 0$ as well.
This completes the proof of the first sentence.

 For the second sentence, by the construction of $U^*$
when $(i_1, i_2) \not\in\Omega$, $A^{-1}(V_{i_1, i_2})^c \cap U^* = U^*$.
We have just shown that $A^{-1}(V_{i_1, i_2})$ has positive
codimension and so  when 
$(i_1, i_2) \in\Omega$, $A^{-1}(V_{i_1, i_2})^c$ is open and
dense in $U^*$. Thus 
$$
\bigcap_{(i_1, i_2)} A^{-1}(V_{i_1, i_2})^c
$$
is open, dense in $U^*$ implying \eqref{claim2}. 
\end{proof}

\subsection{the nonlinear perturbation lemma}\label{Afive}
Now we prove the nonlinear version of  Lemma~\ref{lemma2} which is the last
main step in proving Theorem~\ref{lipschitz}.
\begin{lemma}\label{nonlin}
With $(\hb, \hbw)$ and $n$ 
as above and given $\epsilon > 0$ there exists a $\bw$ with
$\bw\cdot\bm  = 0$, $w_1 = -1$, and  
$\|\bw - \hbw\|<\epsilon$ such that the root $\bb$ to
 $$
F(\bb, \bw) = \One
$$
has the property that $b_i \not= b_j$ when $i\not = j$.
Further, for this $\bw$, the collection of its $(1/n)$-plausible
polynomials has the property that all their associated
solutions  are distinct.
\end{lemma}

\begin{proof}
We show there exist 
$\bw'\in U^*$ arbitrarily close to $0$ so 
that $ g(\bw)$ does not hit $V_{i_1, i_2}$ for
all $i_1 \not = i_2$ which using the definition of
$U^*$  proves the first sentence of the theorem.

In the notation used here Taylor's Theorem says 
\begin{equation}\label{taylor}
g(\bw) = \hbb + \pi(A \bw') + R(\bw)
\end{equation}
where for some $c_1$,  $\|R(\bw)\| < c_1 \|\bw'\|^2$ in a open
set containing $\hbw$ which is compactly contained in $U$.

By Lemma~\ref{lemma2} we may pick a unit vector $\vv'$ in $U^*$ 
pointing into the complement of the
$A^{-1}(V_{i_1 i_2})$ for all $(i_1, i_2) \in \Omega$, see 
Figure~2.

\begin{figure}[h]\label{test}
\includegraphics[width=.7\linewidth]{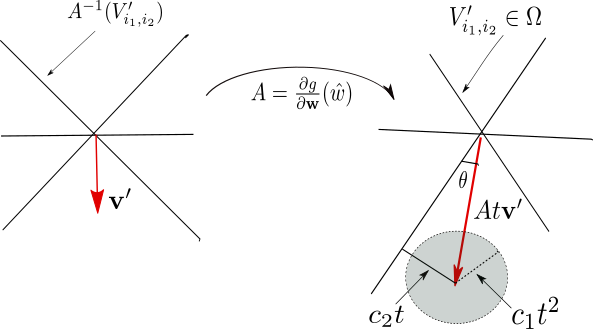}
\caption{Map induced by $g$ on the tangent spaces and error
bounds from Taylor's Theorem.}
\end{figure}

For  $t>0$, $\|A (t \vv')\| = t \|A\vv'\|$. If $\theta$
is the angle $A\vv'$ makes with the nearest $V_{i_1, i_2}'$, then
the distance from $A(t\vv')$ to $V_{i_1, i_2}'$ is $ 
t \|A\vv'\| \sin( \theta) := c_2 t$. Thus if $t$ is small enough
so that  $c_1 t^2 < c_2 t$, then by \eqref{taylor},
$g(t \vv) $ misses all $V_{i_1, i_2}$ for $(i_1, i_2)\in\Omega$.
In addition, shrinking  $t$ if necessary so that
$t \|A\vv'\| + c_1 t^2 < \min\{d(0, V_{i_1,i_2}') :
 (i_1, i_2)\not\in \Omega\}$
ensures that $g(t \vv) $
misses all $V_{i_1, i_2}$ for $(i_1, i_2)\not\in\Omega$. 

Now if $\bw = t\vv$ and  $\bb = g(\bw)$, then  $\bb$ has the property that 
$b_{i_1} \not=  b_{i_2}$ for all $i_1 \not = i_2$ and 
$F(\bb, \bw) = \One$.

Since $F$ was defined using the plausible index sets
of $\hbw$, at this point we know that all the solutions corresponding 
to $p(y) = 1$ are distinct for polynomials arising from
$(1/n)$-plausible index sets for $\hbw$ but using coefficients from
$\bw$. Now we 
 show that if $\bw$ is sufficiently close to $\hbw$ then
these  index sets  are also $(1/n)$-plausible for $\bw$.
The property of a polynomial not having a 
root $b$ in the compact interval $[1/n,1]$ is preserved under
small changes in the coefficients. Since there are finitely many
implausible equations for $\hbw$, for $\bw$ close enough to 
$\hbw$ if an index set is implausible for $\hbw$ it is also implausible for 
$\bw$. Thus index sets associated  with
$\bw$ which are $(1/n)$-plausible for $\bw$  are also $(1/n)$-plausible for 
$\hbw$. So by choosing $t$ small enough in $\bw = t\vv$ above
we guarantee the last statement of the lemma statement.

\end{proof}

\subsection{Proof of Theorem~\ref{lipschitz}}
Given $\phi\in\cSL\setminus A_{p/q, n}^{Lip}$ from
\eqref{Adef} and $\epsilon>0$,
use Lemma~\ref{one} to produce a $\hpsi_\hbw\in\cSL$ 
with derivative of the form $\hpsi_\hbw^\prime = \sum \hw_i I_{X_i}$ with 
$\normL{\phi, \hpsi_\hbw}< \epsilon/2$. If $\hpsi_\hbw$ is not $p/q$-pinching,
for any $b\geq 1/n$, we are done. Otherwise use
Lemma~\ref{nonlin} and \eqref{twolip} to  find a $\bw$ with 
$\psi_\bw\in\cSL$  and  $\normL{\hpsi_\hbw - \psi_\bw}< \epsilon/2$.
For this $\psi_\bw$ by construction, no two of its $(1/n)$-plausible
polynomials $p$ have the same solutions to $p(y) = 1$ in $[1/n,1]$.
 But Lemma~\ref{pincheq}, says that if
$\psi_\bw$ was $p/q$-pinching,
then $\bw$ has three plausible polynomials that have the same
associated solution, a contradiction. Thus  $\psi_\bw\in A_{p/q, n}^{Lip}$,
and so $A_{p/q, n}^{Lip}$ is dense.

\section{PL forcing: $k>2$}\label{notlast}
\subsection{preliminaries}

Recall that a degree one circle homeomorphism $g$ is piecewise linear (PL)
if it is continuous and comprised of a finite number of affine pieces.
The points where the derivative is discontinuous are called
\textit{break points}. The adjectives \textit{ahead} and 
\textit{behind} for  points on the circle are determined
by the usual counterclockwise order on $S^1$. Thus the
interval $(a,b)\subset S^1$ is all the points ahead of $a$ and behind
$b$.  A break point has \textit{type up} if the slope
ahead is larger than that behind, and \textit{type down}
 otherwise. An \textit{interval of definition} is an interval
between adjacent break points. By convention, any statement
about the derivative of a PL function is implicitly followed 
by the phrase ``where is it defined''.

\subsection{Characterization of pinch points}
Recall that a pinch point on the $p/q$-tongue corresponds to 
a $(\omega, b)$ with  $\tf_{b,\omega}^q = R_p$. The first lemma
gives alternative characterizations of such points. The equivalence
of (a) and (b) is standard in the literature, at least for two
break points. The results in this lemma as well as other
characterizations are in \cite{paul}.

\begin{lemma}\label{another} $ \ $
Assume $g$ is a degree one PL-homeomorphism of the circle with
$k>1$ break points and $\rho(g) = p/q$. The following are equivalent
\begin{enumerate}[(a)]
\item $\tg^q = R_{p}$  
\item   Every
 break point of $g$ is on
a periodic orbit that contains at least two break points one of
type up and the other of type down.

\item There exists a PL homeomorphism $h$ with at most 
$\floor{\frac{qk}{2}}$
break points and $h g h^{-1} = R_{p/q}$.
\item $g$ has an invariant measure whose density with respect to
Lebesgue measure is a step function with at most
$\floor{\frac{qk}{2}}$ values, all nonzero.
\end{enumerate}
\end{lemma} 

\begin{proof}
The fact that (c) implies (a) is trivial. Assume the negation of
(b) holds so there is a break point $x_0$ for which 
no element of $g(x_0), g^2(x_0), \dots, g^{q-1}(x_0)$ is a break point
of opposite orientation. Since $g'>0$ away from break points, the left and right
hand derivatives of $q^q$ at $x_0$ are different, so $g^q \not = \Id$. Thus
(a) implies (b). 

Now assume (b). 
Let $\mathcal{P}$ be the collection of $p/q$-periodic orbits of $g$ 
that contain break points. Assume $\mathcal{P}$ contains $d$ elements
and so  
 $d \leq \floor{\frac{qk}{2}}$ and $d = m q$ where $m$
is the number of distinct orbits in $\mathcal{P}$.
  The collection of intervals between points
of $\mathcal{P}$  form an invariant partition of $S^1\setminus \mathcal{P}$. 
Pick a point $z_1\in \mathcal{P}$ and let $z_2, \dots, z_{m+1}\in \mathcal{P}$ 
be the consecutive points  in
$\mathcal{P}$  ahead of  $z_1$.
 Thus $z_{m+1}$ is on the orbit of $z_1$ 
and for $1\leq i \leq m$ there is one $z_i$ from each periodic orbit in 
$\mathcal{P}$. Let $Z_j = [z_j, z_{j+1}]$ and
 $\sZ_j = [(j-1)/m, j/m]\subset S^1$ for $1 \leq j \leq m$. 
For each such $j$ let
$h_{1,j}$ be the affine
homeomorphism from $Z_j$ to $\sZ_j$.  
Now define
$h_{i, j}: g^i(Z_j)\raw R_{p/q}^i(\sZ_j)$ as
 $h_{i,j} = R_{p/q}^i h_{1,j} g^{-i}$.
 Since $g^q = \Id$,
the $h_{i, j}$ fit together into a single homeomorphism
$h$ with  $h g h^{-1} = R_{p/q}$. By construction $h$ is
PL with at most $ d \leq \floor{\frac{qk}{2}}$ break points. 

The equivalence of (c) and (d) is standard. Assuming (c),
the invariant measure
is $ (h^{-1})_* \lambda = h' d\lambda $ and since $h$ is
PL with $d$ break points,  $h'$ is a step function with
$d$ nonzero values. If the density of the
invariant measure in (d) is the step function
$\eta$, then  letting $h(x) = \int_0^x \eta \; d\lambda$ we have
$h_*(\eta d\lambda) = \lambda$. Thus  
$h g h^{-1}$ is a PL homeomorphism that preserves Lebesgue
measure and so is rigid rotation. Since it is conjugate to
$g$ which has rotation number $p/q$, it must be rotation by $p/q$.

\end{proof}

\subsection{Spaces  of PL forcings}

Since a PL map is certainly Lipschitz, using Lemma~\ref{char}
a standard-like PL forcing $\phi$ is one with $\phi' \geq -1$
on all intervals of definition and $\phi' = -1$ on
at least one of them. Let $\mathcal{S}^{PL}_k$ be the
collection of all such forcings with $k$ break points. 
To simplify matters we initially restrict to a normalized class
of forcings which possesses good coordinates.

\begin{definition} $\ $
\begin{enumerate}[1.]
\item Let $\SPLk$ be all PL $\phi\in C^0(S^1)$ with $k$ break points
 such that 
\begin{enumerate}[(a)]
\item $0$ is a break point and $\phi(0) = 0$,
\item $\phi' > -1$ except 
$\phi' = -1$ in the interval of definition just ahead of $0$.
\end{enumerate}

\item 
 Let $\TPLk$ be all
 $(\bw, \bell)\in\RR^{2k}$ such that
\begin{enumerate}[(a)]
\item all $\ell_i > 0$ and $\sum \ell_i = 1$,
\item $w_1 = -1$,  all $w_i > -1$ for $i>1$,  and $w_i \not = w_{i+1}$ 
with the indices mod $k$,
\item $\bw \cdot \bell = 0$.
\end{enumerate}
\end{enumerate}
\end{definition}
Letting $X_1 = [0,\ell_1], X_2 = [\ell_1, \ell_1 + \ell_2],
\dots X_k = [\ell_1 + \dots \ell_{k-1}, 1]$,
the map $(\bw, \bell) \mapsto
\pvl(x) := \int_0^x \sum w_i I_{X_i}(x) \,d\lambda$ is 
a  homeomorphism from $\TPLk$ to $\SPLk$ 
where $\TPLk$ is given 
the topology induced by the standard topology of $\RR^{2k}$
and  $\SPLk$ is given the $C^0$-topology. Starting with
$\phi\in \SPLk$, the corresponding lengths $\bell$ are computed from
the intervals of definition $X_i$ and the weights come from 
$w_i = \phi'$ on
$X_i$. The condition in  $\TPLk$ that $\bw \cdot \bell = 0$
ensures that $\pvl$ is periodic and  $w_i \not = w_{i+1}$
that the ends of intervals are true break points.
Elements of 
of $\SPLk$ and $\TPLk$ are called \textit{reduced} forcings.
 In the main Theorem~\ref{lipgen} below we first show that
nonpinching is generic in $\TPLk$ and then extend it to
all of $\mathcal{S}^{PL}_k$.

Solving the pair of
equations $\sum \ell_i = 1$ and  $\bw\cdot\bell = 0$
 for $\ell_{k-1}$ and $\ell_{k}$ in terms
of $\ell_1, \dots, \ell_{k-2}$ while recalling that
adjacent  $w_i$ are distinct 
yields a local parameterization of $\TPLk$
\begin{equation}\label{param}
(w_2, \dots, w_k, \ell_1, \dots, \ell_{k-2}) \mapsto
\left(-1,w_2, \dots,w_k,\ell_1, \dots, \ell_{k-2}, 
\frac{\eta_1 - w_k \eta_2}{w_{k-1} - w_k},
\frac{w_{k-1} \eta_2 -  \eta_1}{w_{k-1} - w_k}\right)
\end{equation}
where 
\begin{equation*}
\begin{split}
\eta_1 &= \eta_1(\ell_1, \dots, \ell_{k-2})  
= \ell_1 - w_2\ell_2\dots - w_{k-2}\ell_{k-2}\\
\eta_2 &= \eta_2(\ell_1, \dots, \ell_{k-2})  
= 1 - \ell_1 - \ell_2\dots -\ell_{k-2}.
\end{split}
\end{equation*}
Thus $\TPLk\subset\RR^{2k}$ is a $(2k-3)$-dimensional
submanifold without boundary and so, in particular, is
a locally compact Hausdorff space and thus is Baire. Since
$\SPLk$ is homeomorphic to $\TPLk$, it is also Baire.

For $(\bw,\bell)\in \TPLk$ its standard-like family is
\begin{equation*}
f_{b,\omega, \bw, \bell}(x) := f_{b,\omega, \phi_{\bw, \bell}}(x) =
x + \omega + b\; \phi_{\bw, \bell}(x).
\end{equation*}

\subsection{Pinching and configurations}
A 4-tuple $(b, \omega, \bw, \bell)$ with $(\bw, \bell)\in \TPLk$
and the map  
$f_{b,\omega, \bw,\bell}$ are both called
\textit{p/q-pinching} if $\tf_{b,\omega, \bw,\bell}^q = R_p$.
This is equivalent to the pinching at 
$(\omega, b)$ of the $p/q$-tongue for
the forcing $\phi_{\bw,\bell}$.

When $f_{b,\omega,  \bw,\bell}$ is $p/q$-pinching by Lemma~\ref{another} all
the break points are on $p/q$-periodic orbits which contain at least one
other break point. The notion of a configuration captures the combinatorics
of the orbits of break points.
\begin{definition} Given $p/q, k>2$ and $m\leq \floor{q/2}$ 
for $i=1, \dots, m$, let $\sz_i = (1/q) (i-1)/m$. 
A $(p/q, k)$-configuration with  $m$ orbits is the collection
of all $o(\sz_i, R_{p/q})$  coupled
with a collection of marked points $0=\sx_1  < \sx_2 < \dots
< \sx_k$ with at least two on each orbit.

A weighted  $(p/q, k)$-configuration is one with
the additional assignment of a weight $\sw_i$ to each interval
$\sX_i := [\sx_i, \sx_{i+1}]$ with indices mod $k$.
\end{definition}

Again using Lemma~\ref{another} we can assign each $p/q$-pinching 4-tuple  
to a weighted configuration. Assume that there are
$m$ orbits containing break points and recall $0$ is break point.
Let $a$ be the element
in $o(0, f_{b,\omega, \bw,\bell})$ that is
closest to $0$ in the forward direction. For
$i = 1, \dots, m$ let $z_i$ be the point in the orbit
of a break point that is in $[0,a)$ and order the $z_i$ so 
that $0 = z_1 < \dots < z_{m}$. For each $i$, let $\Psi(z_i)
= \sz_i$ and 
$\Psi(f_{b,\omega, \bw,\bell}^j(z_i)) = R_{p/q}^j(\sz_i)$.
If the break points for $f_{b,\omega, \phi_{\bw,\bell}}$ 
are $0=x_1 <  \dots< x_k$
and $j$ and $i$ are  such that $f_{b,\omega, \bw,\bell}^j(z_i) = x_i$,
 let the marked points of the configuration be $\sx_i = R_{p/q}^j(\sz_i)$. 
Thus $\Psi$ gives an order preserving conjugacy from
the $m$  orbits under $f_{b,\omega, \bw,\bell}$
of the break points to the $m$ orbits
of the $\sx_i$ under $R_{p/q}$. Further,
$\Psi$ takes the break points of $f_{b,\omega, \bw,\bell}$
to the marked points of
the configuration or $x_i \mapsto \sx_i$.
 A configuration is called \textit{cyclic} if $m=1$.
To add a weighting, recall that 
$ \phi_{\bw,\bell}' = w_i $ on the interval of definition
$X_i$, and so assign the weight $\sw_i = w_i$ to the interval
$\sX_i$.

The weighted configuration assigned to a $p/q$-pinch
point  $(b, \omega, \bw, \bell)$ has been defined 
so that the values
of $(f_{b,\omega, \bw,\bell}^q)'$ can be read off
from the weighted configuration. For example, if $j_i$ is such
that $R_{p/q}^i([0, \sz_2]) \subset \sX_{j_i}$,
then by  the chain rule and the definition of the weightings,
for  $x$ in the corresponding interval for
$f_{b,\omega, \bw,\bell}$, namely $x \in [0,z_2]$,  then
\begin{equation}\label{compute}
(f_{b,\omega, \bw,\bell}^q)'(x) =
(1 + b w_{j_0})(1 + b w_{j_1}) \dots (1 + b w_{j_{q-1}}).
\end{equation}

It is not clear which unweighted $p/q$-configurations correspond
to actual pinch points. Since we only want to consider those that
do, call a unweighted configuration $C$ \textit{allowable} if
there is a pinch point $(b, \omega, \bw, \bell)$ which corresponds
to $C$. For each $k, p/q$, let $\mathcal{A}_{k, p/q}$ be the set of
allowable $p/q$-configurations with $k$ break points.

A weighted  $(k, p/q)$-configuration induces a set of $m$ degree $q$
polynomials as follows. For each interval $\sZ_i = (\sz_i, \sz_{i+1})$ for
$i=1, \dots, m$ and $j=0, \dots, q-1$ by construction $R_{p/q}^j(\sZ_i)
\subset \sX_{s(j)}$ for some $j$. The $i^{th}$ induced
polynomial is then $p_i(y) = (1 + w_{s(0)} y)(1 + w_{s(1)} y)
\dots (1 + w_{s(q-1)} y)$.

\begin{remark}\label{conclusion}
There is an important conclusion that is required for the rest
of this section. If   
$(b, \omega, \bw, \bell)$ is a $p/q$-pinching 4-tuple for $k$,
then all the polynomials $p_i$ induced by its corresponding weighted
configuration must all have the same solution to $p_i(y) = 1$, namely,
$y=b$. This is because $\tf_{b,\omega, \bw,\bell}^q = R_p$
and so $(f_{b,\omega, \bw,\bell}^q)' = 1$
and the values of the induced polynomials at $b$ are the
derivatives of $f_{b,\omega, \bw,\bell}^q$ in various
intervals.

In the proof of the main Theorem~\ref{lipgen} below we first
fix a configuration.  The space of its allowable
weights whose polynomials do not solve the conditions of the previous
paragraph, and thus are not pinching  is shown to be open and dense. We then
intersect these spaces over the countable set of possible
configurations and use the Baire Theorem to get that
nonpinching is generic.
\end{remark}   
 
\subsection{a perturbation lemma}
Recall from Section~\ref{plaus1}   for a list $\bw\in\RR^k$ with
$-1 \leq w_1 \dots \leq w_k$ and $w_1 < 0 < w_k$ and a given $n>1$
and  $q>1$, 
its set of $(1/n)$-plausible polynomials is  the set of all 
degree $q$ polynomials of the form
$$
p(y) = \prod_{j=1}^q (1 - w_{i_j}y)
$$
with the  $w_{i_j}$ entries in $\bw$ (repeats allowed) with
the property that $p(y)=1$ has a solution $y\in [1/n, 1]$. 
Lemma~\ref{calculus} gives conditions for such a polynomial
to have a solution in  $(0,1)$.

\begin{lemma}\label{basica} $\ $
\begin{enumerate}[(a)]
\item Given $(\hbw,\hbell) \in \TPLk$ with $k>3$, 
$\epsilon > 0$,  $q>1$, and $n>1$ so that the collection
of $(1/n)$-plausible polynomials for $\hbw$ is nonempty,
there exists  $(\obw,\obell) \in \TPLk$,
with   $\|(\obw,\obell) - (\hbw, \hbell)\|<\epsilon$
 so that $\ow_i \not = \ow_j$ when $i\not = j$
and every $(1/n)$-plausible polynomial $p$  for $\obw$ has a different
value of $y$ that solves  $p(y) = 1$.

\item If $C$ is a weighted configuration with weights $\sbw$
so that $\sw_i \not = \sw_j$ when $i\not=j$ and $C$ induces
just one polynomial, then $C$ is cyclic.

\item Given $(\hbw,\hbell) \in \TPLk$ with $k>2$ and $\epsilon > 0$
 there exists $\oell$ with $\|\hbell - \obell\| < \epsilon$ and 
 $(\hbw,\oell) \in \TPLk$ so that $\hell_1/\hell_2 \not = \oell_1/\oell_2$.
\end{enumerate}
\end{lemma}
\begin{proof}
The proof of (a)  follows that of Lemmas~\ref{lemma2} and~\ref{nonlin} and
shares their notation. The main difference is that instead of a vector
of masses $\bm$ there is a vector of lengths $\bell$ and the dimension
of $\bell$ is fixed at $k$ and cannot be adjusted with $q$. 

Recall that $A:T_\hbw\raw T_\hbb$ is the tangent map 
at $\hbw$ and $\hbb$ of the function
 that locally takes a collection of weights $\bw$ to the solutions of
 $(1/n)$-plausible polynomials associated with $\hbw$.
For $(i_1, i_2) \in \Omega$ with the small change to \eqref{noway1} below, 
the proof in Lemma~\ref{lemma2} shows that
  $\myspan(r_{i_1}(A), r_{i_2}(A), E_1)$ is 
three-dimensional in $\RR^k$.
The small change is that in the current situation
$\hw_j = -1$ when $j>1$ is allowable but then
both polynomials  reduce to $(1-y)^q$ and 
$(1-y)^q = 1$ has no solution in $[1/n, 1]$.
Thus since $k>3$, 
$\cO_{i_1, i_2} := \myspan(r_{i_1}(A), r_{i_2}(A), E_1)^c$ is
open, dense, and is a cone in $\RR^k$ and so
$$
\cO = \bigcap_{(i_1, i_2) \in \Omega} \cO_{i_1, i_2}
$$
is also open, dense, and is a cone in $\RR^k$.
So we may pick $\obell\in\cO$ so that all $\oell_i>0$, 
$\sum \oell_i =1$ and $\|\hbell-\obell\| < \epsilon/2$.

Thus the matrix  
\begin{equation*}
B
= \begin{pmatrix}
r_{i_1}(A)\\
r_{i_2}(A)\\
\obell\\
E_1
\end{pmatrix}
\end{equation*}
has rank $4$ for all $(i_1, i_2) \in \Omega$
and so the equation 
\begin{equation*}
  B \vv = \begin{pmatrix} b' \\ b'' \\0 \\ 0 \end{pmatrix}
\end{equation*}
has a solution $\vv$ for arbitrarily small $b'\not= b''$.

Letting $U^*  = U' \cap \obell^\perp \cap W_0$ we 
have that  
for each $(i_1, i_2) \in \Omega$ there is
a vector $\vv' \in U^*$ so that 
$A \vv' \not\in V_{i_1, i_2}'$.
 In addition, with the choice of $U'$ as in Lemma~\ref{lemma2}
when $(i_1, i_2) \not\in\Omega$, 
$A^{-1}(V_{i_1, i_2})^c \cap U^* = U^*$. Thus  
\begin{equation*}
\Gamma= \{ \bw'\in U^* : A \bw' \cap V_{i_1, i_2}' = \emptyset\ 
\text{for all}\ (i_1, i_2)\}
\end{equation*}
is open and dense and there exists a $c>0$ so that
for all nonzero $\bw'\in \Gamma$, $t\bw'\in \Gamma$ for all nonzero $|t|< c$.

Now just as in the proof of Lemma~\ref{nonlin} use Taylor's
Theorem 
to get  $(\oobw, \obell)$ with
the required properties with the exception of  the condition that
 $\oow_i \not = \oow_j$ when $i\not = j$. Note that by the proper
choice of the vector $\bv'$ in the proof of Lemma~\ref{nonlin}
we can ensure that $w_j> -1$ for $j>1$ and so
$(\oobw, \obell)\in\TPLk$.

We require one last perturbation to get a $(\obw, \obell)\in\TPLk$  
with  $\ow_i \not = \ow_j$ while maintaining the condition
on the associated solutions of $(1/n)$-plausible polynomials.
For fixed
$\obell$ and restricted to $\{\bw : w_1 = -1\}$
the collection of $\bw$ such that $(\bw, \obell)\in\TPLk$ is
a subspace of the affine hyperplane $H$ defined by 
$$
w_2 \oell_2 + \dots + w_k\oell_k = \oell_1.
$$
For $i\not = j$, $i,j > 1$ let $V_{i,j} = \{\bw : w_i = w_j\}$.
Now $H$ has normal vector $(\oell_2, \dots, \oell_k)$ while
$V_{i,j}$ has a normal vector  which has a one in
the $i^{th}$ place and a minus one in the $j^{th}$ place and zeros elsewhere.
All of these are not parallel and so each $V_{i,j}$
is  transverse to $H$
and thus
$$L = H \cap \bigcap_{i\not = j} V_{i,j}^c $$
is open and dense in $H$. Thus we may find $\obw$ arbitrarily
close to $\oow$ with the property that all $\ow_i$ are different
and  $\ow_i > -1$ for $i>1$.
As shown at the end of the proof of Theorem~\ref{nonlin} any
$\obw$ close enough to $\oobw$ shares the property that
all its  $(1/n)$-plausible polynomials 
 have  different
associated solutions.
Thus we can perturb $\oobw$ to $\obw$ in $L$ satisfying all the
required conditions.

For (b) assume to the contrary that $C$ is a noncyclic 
configuration which has the property that $\sw_i \not = \sw_j$
 when $i\not=j$ and $C$ induces just one polynomial. Since
$C$ is not cyclic, then $m>1$ and so we have distinct $\sZ_1 = [0, \sz_2]$
and $\sZ_2 = [\sz_2, \sz_3]$. For each $j$ either $R_{p/q}^j(\sZ_1)$
and $R_{p/q}^j(\sZ_2)$ are in adjacent intervals $[\sx_s,\sx_{s+1}]$,
$[\sx_{s+1},\sx_{s+2}]$ or in the same interval. Since by definition
for some $j$, $R_{p/q}^j(\sz_2)$ is a marked point,  the first
possibility must happen at least once. But then since all the weights
 $\sw_i$ are distinct, the list of weights generated by $\sZ_1$ and $\sZ_2$
are different, and so $C$ induces at least two different polynomials,
a contradiction.

For (c), using the parameterization \eqref{param} when 
 $k\geq 4$ we can alter $\ell_1$ and $\ell_2$
independently while fixing
$\bw$. Since $(\hbw,\hbell) \in \TPLk$ we can 
thus perturb, say $\hell_1$, while maintaining all $\ell_i>0$
and $\sum \ell_i = 1$.
 When $k=3$, treating $\ell_2$ as a function of $\ell_1$,
\begin{equation*}
\frac{\ell_2}{\ell_1} = 
\frac{\ell_1 - w_3(1-\ell_1)}{\ell_1(w_2 - w_3)}
\end{equation*}
and its derivative with respect to $\ell_1$ 
is 
\begin{equation*}
\frac{w_3}{\ell_1^2 (w_2 - w_3)} \not = 0.
\end{equation*}
So we can perturb $\ell_1$ while 
maintaining $\ell_2 = (\ell_1 - w_3(1-\ell_1))/(w_2 - w_3)$
to alter $\ell_1/\ell_2$
while staying in $\TPL_3$.

\end{proof}

\subsection{The genericity theorem}

\begin{theorem}\label{lipgen}
For each $k>2$ it is generic amongst all standard-like forcings
with $k$ break points that there is no pinching
in any of the  rational Arnol'd tongues in its $(\omega, b)$ ($b>0$) 
family.
\end{theorem}

\begin{proof}

We first prove genericity in $\TPLk$ and then extend it  to
 $\mathcal{S}^{PK}_k$.

Fix $p/q$ and  $n>2$ and an allowable configuration 
$C\in \mathcal{A}_{k, p/q}$.  Let $F_{ p/q, C, n}$ be
all  $(\bw, \bell) \in \TPLk$ such that
 there exists
$b\in [1/n, 0]$ and  $\omega \in S^1$ 
so that $(b,\omega, \bw, \bell)$ 
is a $p/q$-pinch point 
which  induces the configuration $C$.
We show its complement 
$G_{ p/q, C, n} := \TPLk\setminus F_{p/q, C, n} $ 
is open and dense in $\TPLk$. Since convergence in
$\TPLk$  implies $C^0$ convergence in $\mathcal{S}^{PL}_k$,
 as in the proof of Theorem~\ref{gen1},
$F_{p/q,C , n}$ is closed in $\TPLk$.

For density, pick $(\hbw, \hbell)\in F_{p/q, C, n}$ so that
$ (\hb,\homega, \hbw, \hbell) $ is a $p/q$-pinch point with configuration
$C$. We need to show that there are $(\obw, \obell)\in \TPLk$
arbitrarily close to $(\hbw, \hbell)$ with the property that
 for all $(\omega, b)$, the 4-tuple
 $(b,\omega, \obw, \obell) $ is not a $p/q$-pinch point with
configuration $C$.

There are two cases to consider. First assume that $C$ is not
cyclic. Now each orbit of a marked point in a configuration
must contain another marked point so 
when $k=3$ every configuration is cyclic. We
therefore assume $k>3$. Use Lemma~\ref{basica}(a) to perturb 
$(\hbw,\hbell)$ to $(\obw, \obell)$  with the
property that all the $(1/n)$-plausible polynomials for $\obw$ have
different roots and all the elements of $\obw$ are distinct.
If $ (b,\omega, \obw, \obell) $ for all $b, \omega$
is not a $p/q$-pinch point with configuration $C$
we are done, so assume $ (\ob,\oomega, \obw, \obell) $ is. 
Since all the $(1/n)$-plausible
equations for  $\obw$ have different roots  and one of them
must have solution $\ob$, using Remark~\ref{conclusion}
we see that the configuration $C$
has just one induced equation and further by the choice of $\obw$ 
all the elements of $\obw$ are distinct. Thus by Lemma~\ref{basica}(b),
$C$ is cyclic, a contradiction and so 
$(\obw, \obell)\in G_{ p/q, C,  n}$.

The second case is when $C$ is cyclic. Using the given pinch point
$(\hb,\homega, \hbw, \hbell)$ assign the usual weighting 
$\sw_i = \hw_i$ to the 
configuration $C$. We first perform some
calculations involving just the weighted
configuration and then connect the results to the pinch point maps. 

Let $\sJ = [0, 1/q]$ and $\sw_{n_0}, \dots, \sw_{n_{q-1}}$ be the
weights with $R^i_{p/q}(\sJ) \subset \sX_{n_i}$. For $N = 0, \dots, q-1$
let 
\begin{equation*}
Q_N(y) = \prod_{i=0}^N (1 + \sw_{n_i} y).
\end{equation*}
In addition,
let $\{i_\alpha\}$  and $\{\oi_\beta\} $ and be such that 
\begin{equation}\label{union}
\begin{split}
\sX_1 &= \bigcup_{\alpha} R^{i_\alpha}_{p/q}(\sJ)\\ 
\sX_2 &= \bigcup_{\beta} R^{\oi_\beta}_{p/q}(\sJ).
\end{split}
\end{equation}
 This is possible because
 $C$ is cyclic. Finally, let
\begin{equation*}
\begin{split}
P_1(y) &= \sum_\alpha Q_{i_\alpha}(y)\\
P_1(y) &= \sum_\beta Q_{\oi_\beta}(y) 
\end{split}
\end{equation*}

We now return to considerations of the pinch point
$  (\hb,\homega, \hbw, \hbell) $. Using the connection of the
weighted configuration to the map, the  coupling
$\hb$ must be the unique solution to  $Q_{q-1}(y)=1$ in $(0,1)$.
 In addition,
if $\hJ = [0,\ha]$ where $\ha$ is the point on the orbit of $0$
closest to $0$ in the positive direction, then as a consequence
of \eqref{union}
\begin{equation*}
\begin{split}
\hX_1 &= \bigcup_{\alpha} f^{i_\alpha}_{\hb,\homega, \hbw, \hbell}(\hJ)\\ 
\hX_2 &= \bigcup_{\beta} f^{\oi_\beta}_{\hb,\homega, \hbw, \hbell}(\hJ).
\end{split}
\end{equation*}
where $\hX_1 = [0, \hx_1]$ and $\hX_2 = [\hx_2, \hx_3]$.
Thus by \eqref{compute}, 
\begin{equation*}
\begin{split}
\hell_1 &= |\hX_1| =  P_1(\hb) |J|\\ 
\hell_2 &= |\hX_2| =  P_2(\hb) |J|,
\end{split}
\end{equation*}

Now using Lemma~\ref{basica}(c) 
perturb $\hbell$ to $\obell$ so that $\hell_1/\hell_2 \not =
\oell_1/\oell_2$ while keeping $(\hbw, \obell) \in \TPLk$.
 If  $(b, \omega, \hbw, \obell)$ is not a pinch point
with configuration $C$ for all $b,\omega$ we are done, so assume it 
$ (\ob, \oomega, \hbw, \obell)$ is. Since 
$f_{\ob, \oomega, \hbw, \obell}$  and 
$f_{\hb, \homega, \hbw, \hbell}$ have the same weights  
and the same configuration, they have the same
 weighted configurations. 
 Thus $\ob$ is also the unique
solution in $(0,1)$ of  $Q_{q-1}(y)=1$ and so $\ob = \hb$. Further,
adjusting the definition of $\hJ$, $\hX_1$ and $\hX_2$
to those of $f_{\ob, \oomega, \hbw, \obell}$,
 then \eqref{union} also holds implying
\begin{equation*}
\begin{split}
\oell_1 &= |\oX_1| =  P_1(\ob) |\oJ|\\ 
\oell_2 &= |\oX_2| =  P_2(\ob) |\oJ|.
\end{split}
\end{equation*}
Now using the fact that $\ob = \hb$, $\oell_1/ \oell_2 = \hell_1/\hell_2$,
this is a contradiction. Thus the perturbation $(\obw,\obell) \in
G_{ p/q, C, m, n}$ as needed.

Every $p/q$-pinch point must correspond to some allowable 
$p/q$-configuration, and so the collection of reduced standard-like forcings
with $k$ break points which have  no pinching
in any of their rational Arnol'd tongues is 
\begin{equation}\label{yay}
G = \bigcap_{p/q, C\in\mathcal{A}_{k,p/q}, n} G_{ p/q,C, n}
\end{equation}
which is dense, $G_\delta$ in $\TPLk$ by the Baire Category Theorem.

The last step is to extend the result to $\mathcal{S}^{PL}_k$. Define
$$
\Phi: \RR\times S^1 \times \TPL_k\raw \mathcal{S}^{PL}_k
$$
via 
$\Phi(r,s, \bw, \bell) = r + \phi_{\bw,\bell}\circ T_s$,
and so $\Phi$ is continuous. 

Recall that elements of $\TPLk$ were not allowed to have two 
intervals having weight $-1$ and so correspondingly
$\phi\in\SPLk$ couldn't have two intervals of definition 
where $\phi' = -1$. These are accounted for as follows.
Let
$$
\mathcal{C}_k = \{\psi \in \mathcal{S}^{PL}_k:  
\psi' = -1 \ \text{on at least 2 intervals of definition}\}.
$$
We show that the image of $\Phi$ is
 $\mathcal{S}^{PL}_k\setminus \mathcal{C}_k$. Since
$(\bw, \bell)\in \TPLk$ has exactly one weight of
minus one, clearly the image of $\Phi$ is 
contained in $\mathcal{S}^{PL}_k\setminus \mathcal{C}_k$.
Now if 
$\psi\in \mathcal{S}^{PL}_k\setminus \mathcal{C}_k$, 
then $\psi' = -1$ on a unique interval
of definition. If this 
interval has left endpoint $p$, then 
$-\psi(p) + \psi(x+p)\in\SPLk$ and so 
for a unique  $(\bw, \bell)$, we have
$-\psi(p) + \psi(x+p)= \phi_{\bw, \bell}$.
Thus 
\begin{equation*}
\Phi(\psi(p), -p, \bw, \bell ) = \psi.
\end{equation*} 
These choices are unique and depend continuously on $\psi$, thus
$\Phi$ is a homeomorphism 
$\RR\times S^1 \times \TPL_k\raw \mathcal{S}^{PL}_k\setminus \mathcal{C}_k$.

We showed above that $G$ defined in \eqref{yay} is
dense, $G_\delta$ in $\TPLk$ and so 
$G' = \RR\times S^1 \times G$ is dense, $G_\delta$
in $\RR\times S^1 \times \TPLk$. Thus
 $\Phi(G')$ is
dense, $G_\delta$ in $\mathcal{S}^{PL}_k\setminus \mathcal{C}_k$.
 But the set removed $\mathcal{C}_k$
is closed, and nowhere dense in the whole set $\mathcal{S}^{PL}_k$ and
so $\tG = \Phi(G')$ is also
dense, $G_\delta$ in the whole set $\mathcal{S}^{PL}_k$.

We finally find the consequences for pinching.
Using the underlying functions,
let $\psi = \Phi(r,w,\phi)$  and so $\psi = r + \phi\circ R_s$.
Thus
$$
f_{b,\omega, \psi} = R_s^{-1} f_{b,br + \omega, \phi} R_s
$$
and so 
$\rho(f_{b,\omega, \psi}) = \rho( f_{b,br + \omega, \phi})$.
Therefore the $(\omega, b)$ diagram for $\psi$ is obtained by the linear
transformation $(\omega,b) \mapsto (br + \omega, b)$  from that
of $\phi$. This 
does not alter pinching of the rational
tongues. Thus since all $(\hw, \hell)\in G$ have
no pinching, all $\psi$ in the dense, $G_\delta$ set $\tG$ in
$\mathcal{S}^{PL}_k$ have no pinching in any of their rational
tongues.
\end{proof}

\begin{remark}
It is worth noting why the above proof does not work when there
are $k=2$ break points. The  second part of
the proof uses  Lemma~\ref{basica}(c) to 
perturb
$\bell$ and leaves $\bw$ alone, all the while remaining in $\TPLk$.
When $k=2$ the equations defining $\TPLk$ imply that 
$\ell_1 = 1 - \ell_2$ and $(1 + w_2)\ell_2 = 1$ and so $\bell$
and $\bw$ can't be perturbed independently while remaining in 
$\TPL_2$.
\end{remark}

\section{PL forcing: $k=2$}\label{last}
As illustrated by the theorem of Campbell, Galeeva, Tresser, and Uherka
the case of $k=2$ break points is quite different than $k>2$. 
In particular, there are infinitely many  pinched rational
tongues in the $(\omega, b)$ diagram of every standard-like $\phi$.
For completeness we give an alternative proof of their theorem using
the methods of this paper.

\subsection{preliminaries}

The equations defining $\TPL_2$ are $\ell_1, \ell_2 > 0$,
$w_2 > -1$ and 
\begin{equation}\label{newpair}
\begin{aligned}
-\ell_1 + w_2 \ell_2 &= 0\\
\ell_1 + \ell_2 &= 1.
\end{aligned}
\end{equation}
Solving, $\ell_1 = 1-\ell_2$ and  $(1 + w_2) \ell_2 = 1$, so there
is  a one-parameter family of reduced standard-like forcings
with two break points. 
In the context of the methods in
this paper we use $w_2$ as the parameter
and call it just $w$. Note that \eqref{newpair} forces
$w> 0$.  We  denote the corresponding
forcing as $\phi_w \in \SPL_2$. The standard-like family
is
$$
f_{b,\omega, w}(x) = x + \omega + b\; \phi_w(x).
$$
The break points for $\phi_w$ and $f_{b,\omega, w}$ are
$0$ and $\ell_1 =\frac{w}{w+1}$, which we denote as $x_u$ and $x_d$
respectively.
A $p/q$-pinch point is now a triple $ (b, \omega, w)$ 
with $\tf_{b, \omega, w} ^q = R_p$.

As in the previous sections the polynomials associated
with the derivative of $f_\eta^q$ are the main tool.
For $j=0, \dots, q$ let
\begin{equation}\label{thepoly}
p_{q,j, w}(y) = (1-y)^j (1+w y)^{q-j}.
\end{equation}
Using Lemma~\ref{calculus}, $p_{q, j, w}(y) = 1$ has a solution
$y\in(0,1)$  if and only if 
$-j + w(q-j)>0$ or using \eqref{newpair}, 
$ 0 < j \leq  \tau(q, w)$ where $\tau(q, w) = 
\ceiling{\frac{w}{w+1}q}-1 =\ceiling{\ell_1 q} - 1$. 
For $j = 1, \dots, \tau(q,w)$ let $b_{q,j, w}$ be the
unique root of $p_{q, j, w}(y) = 1$ in $(0,1)$ as given
by Lemma~\ref{calculus}.

For future reference 
note that  for a fixed $b\in (0,1)$ the values 
$p_{q, j, w}(b)$ are reverse ordered by $j$ or
\begin{equation}\label{reverse}
p_{q, 1, w}(b) > p_{q, 2, w} (b) > \dots  > p_{q, \tau(q), w}(b).
\end{equation}

We now connect $\eqref{thepoly}$ to the derivative
of $f_{b, \omega, w}$ using a standard address/itinerary scheme.
Let  $X_{-1}$ be the interval in the circle counter-clockwise from
$x_d$ to $x_u$  and $X_{1}$ be the interval counter-clockwise from
$x_u$ to $x_d$. Thus
if  $x\in X_{-1}$ then $\phi_w'(x) = -1$ and
 $f_{b, \omega, w}'(x) = 1- b$ and if  $x\in X_{1}$ then
 $\phi_w'(x) = w$ and
 $f_{b, \omega, w}'(x) = 1+bw$.

For $x$ with $\{x, \dots, f_{b, \omega, w}^{q-1}(x)\} \cap \{x_d, x_u\}
 = \emptyset$
define $\iota(x) \in \{-1, 1\}^q$ by $\iota(x)_i = j$ if 
$f_{b, \omega, w}^i(x)\in X_j$
and let $\gamma(x)$ be the number of times $-1$ occurs in $\iota(x)$. Thus
using the chain rule
\begin{equation}\label{chain2}
(f_{b, \omega, w}^q)'(x) = 
(1-b)^{\gamma(x)} (1+w b)^{q-\gamma(x)}= p_{q, \gamma(x), w}(b).
\end{equation}

Let $\Lambda_j = \{ x : \gamma(x) = j\}$ and so $\Lambda_j$ is a finite union
of intervals and $(f_{b, \omega, w}^q)'(x) = p_{q,j, w}(b)$ 
when $x\in \Lambda_j$.
 Let $\lambda_j = \lambda(\Lambda_j)$. Now $\cup \Lambda_j$ is the entire
circle minus a finite set of points and 
since $f_{b, \omega, w}^q$ is a degree one circle map, 
$1 = \int_{S^1} (f_{b, \omega, w}^q)' \; d\lambda$. Thus
the $p_{q,j, w}(b)$ and $\lambda_j$ satisfy a system of equations
\begin{equation}\label{system}
\begin{aligned}
\sum_{j=0}^q \lambda_j &= 1\\
\sum_{j=0}^q p_{q,j,w}(b) \lambda_j &= 1
\end{aligned}
\end{equation}

\subsection{The pinching theorem}

\begin{lemma}\label{lastlem}
If $b = b_{q, j, w}$ for some $1 \leq j \leq \tau(q)$ and 
 for some $\omega$, 
$x_d$ is a $p/q$-periodic orbit for  $f_{b, \omega,w}$, then
$(b,\omega, w)$ is a $p/q$-pinch point.
\end{lemma}
\begin{proof}

The orbit of $x_d$ divides the circle into an invariant partition by
intervals. The point $x_u$ is in one of them, say $X_0$.
 If it is an endpoint, then both break points are on the same
periodic orbit and so $(b,\omega, w)$ is a pinch point by 
Lemma~\ref{another}.

The second case is that $x_u$ is in the interior of $X_0$. Let
$B_{-1} = X_{-1}\cap X_0$ and $B_{1} = X_{1}\cap X_0$. Thus
for $x_{-1}\in B_{-1}$, $\iota(x_{-1})_0 = -1$ and for  
$x_1\in B_{1}$, $\iota(x_1)_0 = 1$.
Since the partition elements are permuted, for $1 \leq j  < q$, 
$\iota(x_1)_j = \iota(x_{-1})_j$. Thus letting $\gamma :=  \gamma(x_{1})$,
we have $\gamma(x_{-1}) = \gamma+1$ and 
so by \eqref{chain2} for all $x\in S^1$, $(f_{b,\omega, w}^q)'(x)$ is either 
$p_{q, \gamma, w}(b)$ or $p_{q, \gamma+1, w}(b)$.

Since there are only two nonempty $\Lambda_j$, the system of equations
\eqref{system} reduces to
\begin{equation*}
\begin{aligned}
 \lambda_\gamma + \lambda_{\gamma+1}  &= 1\\
 p_{q,\gamma,w}(b) \lambda_\gamma + p_{q,\gamma+1,w}(b) \lambda_{\gamma+1}  &= 1
\end{aligned}
\end{equation*}
where we must have $\lambda_\gamma, \lambda_{\gamma+1} \geq 0$
since they are masses of sets. Letting $p_\gamma = p_{q,\gamma,w}(b)$ 
 the system  has solution
\begin{equation*}
\lambda_\gamma = \frac{1 - p_{\gamma+1 }}{p_\gamma -  p_{\gamma+1} },
\ \ 
 \lambda_{\gamma+1} = \frac{p_{\gamma}-1 }{p_\gamma -  p_{\gamma+1} }.
\end{equation*}
Recall now \eqref{reverse} and so  $p_\gamma -  p_{\gamma+1}>0$. There
are now a number of cases constrained by \eqref{reverse}. Recall
that the $b = b_{q, j,w}$. If $\gamma > j$, then $p_\gamma < p_j = 1$
and so $\lambda_{\gamma+1} < 0$, a contradiction. Similarly, 
$\gamma+1 < j$ implies the contradiction $\lambda_{\gamma} < 0$.
The last cases are $\gamma+ 1 > \gamma = j$ and $j = \gamma+1 > \gamma$.
Since $p_j =1$, the first case implies $\lambda_{\gamma+1} = 0$
and the second that $\lambda_{\gamma} = 0$, so in both cases 
 only $\Lambda_j$ is nonempty.

Thus in all cases we have shown that the hypothesis imply that
$(f_{b,\omega, w}^q)' \equiv 1$. Since 
$f_{b,\omega, w}$ has a $p/q$-periodic orbit this shows that
$(b, \omega, w)$ is a $p/q$-pinch point.
\end{proof}

\begin{definition}
Given $p/q$ for $j$ with $1 \leq j \leq \tau(q,w)$
and let $\omega_{p/q, j, w} = 
\min\{\omega : \rho(f_{b_{q,j,w}, \omega, w}) =
p/q\}$.
\end{definition}

Theorem $4$ in \cite{trouble} uses $\delta = \ell_1$ and so
we will use  that in our theorem statement; from \eqref{newpair}
$\delta = \frac{w}{w+1}$. 

\begin{theorem}[Campbell, Galeeva, Tresser, and Uherka] \label{didit}
Given $(w, \bell)\in \TPL_2$,
for each $1 \leq j \leq \ceiling{\delta q}-1$, with
 $\omega_{p/q, j, w}$ and $b_{q,j, w}$
defined as above,  
$(b_{q,j, w}, \omega_{p/q, j,w}, w)$ is a $p/q$-pinch point. 
Thus  each $T_{p/q}$ in the $(\omega, b)$ diagram 
for $\phi_\bw$ pinches in $\ceiling{\delta q}-1$ 
places where $\delta$ is the width of the region where $\phi_w' = -1$.
\end{theorem}
\begin{proof} For simplicity of notation let
$L = f_{b_{q,j,w}, \omega_{p/q, j,w}, w}$ 
Now $(\omega_{p/q, j},b_{q,j} )$ is on the left
boundary of $T_{p/q}$ and so $\tL^q \leq R_p$
with equality at some point $x_0$. Since the graph of 
$\tL^q$ is below the graph of $R_p$,
the slope at $x_0$ must decrease from left to right as $x$ increases.
Since $L' > 0$ everywhere it is defined,
this can only happen at $x_0$ if $x_0 = L^{j}(x_d)$
for some $0 \leq j < q$. But since $\tL^q(x_0)
= x_0 + p$, we have that $x_d$ is a $p/q$-periodic orbit
for   $L$. Thus by Lemma~\ref{lastlem},
$(b_{q,j, w}, \omega_{p/q, j,w}, w)$ is a $p/q$-pinch point.
\end{proof} 
Theorem~\ref{pinchmain} in the Introduction follows
by extending this result to all of $\SPL_2$ 
as in the end of the proof of Theorem~\ref{lipgen}. 

\medskip

\noindent\textbf{Declarations}

\smallskip

\textbf{Data availability statement:} 
No datasets were generated or analysed during the current study. 

\smallskip

\textbf{Funding and Competing interest:} 
The author received a Professor's salary from the University of
Florida during a portion of the research for and
preparation of this paper. There was no other funding.

\bibliography{pinch}{}
\bibliographystyle{plain}

\end{document}